\documentclass[12pt]{amsart}
\usepackage{fullpage}
\usepackage{color}
\usepackage{pstricks,pst-node,pst-plot} 
\usepackage{graphicx,psfrag} 
\usepackage{color} 
\usepackage{tikz}
\usetikzlibrary{decorations.markings, hobby, knots}
\usepackage{pgffor}
\usepackage{hyperref}
\usepackage{todonotes}
\usepackage{subfigure}
\usepackage{mathtools}
\usepackage{verbatim}

\usepackage{bm}
\usepackage{multirow}

\usepackage{perpage}	 
	
\MakePerPage{footnote}	 

\newtheorem{theorem}{Theorem}[section]
\newtheorem*{theorem-non}{Theorem}
\newtheorem{lemma}[theorem]{Lemma}

\newtheorem{corollary}[theorem]{Corollary}
\newtheorem{proposition}[theorem]{Proposition}

\theoremstyle{definition}
\newtheorem{definition}[theorem]{Definition}
\newtheorem{example}[theorem]{Example}

\newtheorem{fact}[theorem]{Fact}

\theoremstyle{remark}
\newtheorem{remark}[theorem]{Remark}

\numberwithin{equation}{section}

\definecolor{gray}{rgb}{.5,.5,.5}

\definecolor{black}{rgb}{0,0,0}

\definecolor{blue}{rgb}{0,0,1}
\def\blue{\color{blue}}
\definecolor{red}{rgb}{1,0,0}

\definecolor{green}{rgb}{0,1,0}

\definecolor{yellow}{rgb}{1,1,.4}

\newrgbcolor{purple}{.5 0 .5}
\newrgbcolor{black}{0 0 0}
\newrgbcolor{white}{1 1 1}
\newrgbcolor{gold}{.5 .5 .2}
\newrgbcolor{darkgreen}{0 .5 0}
\newrgbcolor{gray}{.5 .5 .5}
\newrgbcolor{lightgray}{.75 .75 .75}
\newrgbcolor{lightred}{.75 0 0}

\def\tot{\operatorname{tot}}

\newcommand{\crosspos}{
\begin{tikzpicture}[baseline=-2]
\draw[white,line width=1.5pt,double=black,double distance=.5pt] (0,0.2) -- (0.3,-0.1);
\draw[white,line width=1.5pt,double=black,double distance=.5pt] (0,-0.1) -- (0.3,0.2);
\end{tikzpicture}}

\newcommand{\crossposor}{
\begin{tikzpicture}[baseline=-2]
\draw[white,line width=1.5pt,double=black,double distance=.5pt] (0,0.2) -- (0.3,-0.1);
\draw[<-] (0,0.2) -- (0.3,-0.1);
\draw[white,line width=1.5pt,double=black,double distance=.5pt] (0,-0.1) -- (0.3,0.2);
\draw[->] (0,-0.1) -- (0.3,0.2);
\end{tikzpicture}}

\newcommand{\crosssmooth}{
\begin{tikzpicture}[baseline=-2]
\draw[white,line width=1.5pt,double=black,double distance=.5pt, rounded corners] (0,0.2) -- (.1,.05) -- (0,-0.1);
\draw[white,line width=1.5pt,double=black,double distance=.5pt, rounded corners] (0.3,-0.1) -- (.2,.05) -- (0.3,0.2);
\end{tikzpicture}}

\newcommand{\mirror}{\textsf{mirror}}
\newcommand{\merge}{\textsf{merge}}

\begin{document}

\title{Average signature and 4-genus of 2-bridge knots}

\author{Moshe Cohen}
\address{Mathematics Department, State University of New York at New Paltz, New Paltz, NY 12561}
\email{cohenm@newpaltz.edu}

\author{Adam M. Lowrance}
\address{Department of Mathematics and Statistics, Vassar College, Poughkeepsie, NY 12604}
\email{adlowrance@vassar.edu}

\author{Neal Madras}
\address{Department of Mathematics and Statistics, York University, Toronto, ON M3J 1P3}
\email{madras@yorku.ca}

\author{Steven Raanes}
\address{Department of Mathematics, The Ohio State University, Columbus, OH 43210}
\email{raanes.1@osu.edu}

\begin{abstract}
We show that the average or expected absolute value of the signatures of all 2-bridge knots with crossing number $c$ approaches $\sqrt{{2c}/{\pi}}$.  
Baader, Kjuchukova, Lewark, Misev, and Ray consider a model for 2-bridge knot diagrams indexed by diagrammatic crossing number $n$ and show that the average 4-genus is sublinear in $n$.  We build upon this result in two ways to obtain an upper bound for the average 4-genus of a 2-bridge knot:  our model is indexed by crossing number $c$ and gives a specific sublinear upper bound of $9.75c/\log c$.
\end{abstract}


\maketitle

\section{Introduction}
\label{sec:intro}

Many knot theory results read as follows:  here is an infinite family of knots whose one chosen knot invariant has some behavior and whose other chosen invariant has another behavior \cite{DiaErnThi,BriJen, LudEvaPaa, LowAltDist,BakMot,Mis,FucPurSte}.  In contrast, the present work is concerned with the average or expected behavior of a given knot with $c$ crossings, answering a different kind of question: we study the average signature and $4$-genus of the set of $2$-bridge knots with crossing number $c$.

One approach to  studying the typical behavior of a knot, not considered in this work, is to feed already-existing data about small knots into machine learning artificial intelligence models in order to see what connections can be found \cite{Hug:ml,machinelearning, LevHajSaz}.  Another approach, also not considered here, is to develop a random model for all knots from which this information can be computed directly \cite{EZHLN, EZ:survey, OwaTsv}.  However, it may be the case that certain models favor certain behaviors, as in \cite{McC}, and thus much attention is still needed for these efforts to be completely successful.

The approach of the current work may help to establish a baseline for the problem above by determining a point of comparison for random models:  do the behaviors of the random model match the behaviors of the set of all knots?  Certainly if the behaviors of the set of all knots could be easily understood, there would be no need for a random model.  Because this seems unlikely in general, we restrict our attention to the set of 2-bridge knots, for which we are able to better understand various behaviors.

We continue in the same spirit as \cite{Coh:lower, CohLow,CohLowURSI} by using a combinatorially rich method of tabulating $2$-bridge knots (also see \cite{Co:3-bridge, CoKr, CoEZKr} for further background).  One advantage of this setting is the abundance of recursive identities that underpin this work. Utilizing this setup, the first two authors computed in \cite{CohLow} the average Seifert genus of a 2-bridge knot, which was independently proven by \cite{SuzTra:genus} and confirmed experimentally by \cite{RayDiao}.  The first, second, and last authors of the present paper with others went on to show in \cite{CohLowURSI} that the distribution of genera for 2-bridge knots with crossing number $c$ approaches a normal distribution as $c$ approaches infinity.  The proof is technical in places, leading the authors to more questions than answers.  Does the result hold in general for all knots?  Does this follow from properties of genus like additivity?  
Do similar combinatorial random model results for other topological objects imply that 2-bridge knots can be seen as combinatorial? Do the distributions of other knot invariants over a fixed crossing number approach normal distributions as $c$ approaches infinity, and can we use these to say things about finite cases? 

Continuing in this vein, the second author with undergraduate students \cite{LowURSI:braids} and also Suzuki and Tran \cite{SuzTra:braids} independently computed the average braid index of 2-bridge knots.

\subsection{Signature and 4-genus}  
The \textit{signature} $\sigma(K)$ of the knot $K$, defined by Trotter \cite{Trotter}, is the difference between the number of positive and negative eigenvalues of $V+V^T$ where $V$ is any Seifert matrix of $K$.  Gordon and Rick Litherland (1953-2022) expressed the signature of a knot as the signature of the Goeritz matrix up to a diagrammatic `correction term' \cite[Theorem 2]{GorLit}.  Gallaspy and Jabuka  started \cite{GalJab} and Qazaqzeh, Al-Darabsah, and Quraan completed \cite{QazAlDQur} an analysis of this correction term  based on the parities of the number of crossings in each twist region.

Traczyk \cite{Tra} used the Gordon-Litherland method of computing knot signature to find a diagrammatic formula for the signature of alternating links. Specifically, if $D$ is an alternating diagram of a knot where $s_A$ is the number of circles in the all-$A$-smoothing of the knot diagram and where $c_+$ is the number of positively-oriented crossings, then $\sigma(K)=s_A - c_+ -1$, which we use here.  One benefit of using our approach below is that the orientations of the crossings can be easily determined, as in \cite{Coh:lower}, avoiding the difficulties of \cite{GalJab} and \cite{QazAlDQur}.

Dasbach and the second author continued to study signature for links with Turaev genus one \cite{DasLow:gT1}.  Ville and Soret \cite{SorVil} studied closures of certain random quasipositive 3-braids, including Lissajous toric knots, showing these have signature 0.  Dunfield and Tiozzo \cite{DunTio} studied closures of random positive 3-braids and are able to prove a central limit theorem on their signatures using Bj\"orklund-Hartnick's central limit theorem for quasimorphisms \cite{BjoHar}.

The \emph{4-genus} $g_4(K)$ of a knot $K$ is the minimum genus of a locally-flat surface embedded in the 4-ball whose boundary is the knot.  A knot is \emph{slice} if its $4$-genus is zero.  In recent years, much attention has been paid to the question of whether a knot is slice \cite{Pic:slice,Hom:slice,Lev:slice}.

Two inequalities involving the $4$-genus of a knot will be important in this article. First, Murasugi \cite{Mur:4genus} proved that half of the absolute value of the signature of a knot is a lower bound on the $4$-genus of the knot, that is, he showed 
\begin{equation}
    \label{eq:sigbound}
    \frac{|\sigma(K)|}{2} \leq g_4(K)
\end{equation}
for any knot $K$. The second inequality is an upper bound on the $4$-genus of a knot coming from the minimum crossing number $c(K)$ of a knot. By definition, the $4$-genus of a knot is bounded from above by the $3$-genus $g_3(K)$ of the knot, and applying Seifert's algorithm to a minimum crossing diagram of the knot implies $g_3(K)<\frac{c(K)}{2}$. Therefore for any knot $K$
\begin{equation}
    \label{eq:crossing}
    g_4(K) < \frac{c(K)}{2}.
\end{equation}

Recently Baader, Kjuchukova, Lewark, Misev, and Ray \cite{BKLMR} show that the average 4-genus of a 2-bridge knot diagram with $2n$ crossings is sublinear with respect to $n$.  In Section \ref{sec:4-genus}, we modify their proof techniques for our setting in order to say something more precise about this sublinearity.

\subsection{Results and Organization}
Let $\mathcal{K}(c)$ be the set of $2$-bridge knots of crossing number $c$ where only one of each chiral pair is an element of $\mathcal{K}(c)$. Define the \textit{average of the absolute value of the signature} $\overline{|\sigma|}(c)$ of the set of 2-bridge knots with crossing number $c$ by
\[\overline{|\sigma|}(c) = \frac{\sum_{K\in\mathcal{K}(c)} |\sigma(K)|}{|\mathcal{K}(c)|}.\]
Our first main theorem asymptotically computes $\overline{|\sigma|}(c)$.
\begin{theorem}
    \label{thm:avgsignature}
    The average absolute value of the signature $\overline{|\sigma|}(c)$  satisfies the limit 
    \[\lim_{c\to\infty}\left( \overline{|\sigma|}(c) - \sqrt{\frac{2c}{\pi}}\right) = 0.\]
\end{theorem}

Define the \textit{average 4-genus} $\overline{g_4}(c)$ of the set of 2-bridge knots with crossing number $c$ by 
\[\overline{g_4}(c) = \frac{\sum_{K\in\mathcal{K}(c)} g_4(K)}{|\mathcal{K}(c)|}.\]
Our second main theorem gives an explicit sublinear upper bound for $\overline{g_4}(c)$.
\begin{theorem}
\label{thm:avg4genus}
For all $c\geq 3$, the average 4-genus $\overline{g_4}(c)$ is bounded above by $\displaystyle\frac{9.75c}{\log c}$ and hence is sublinear in the crossing number $c$.
\end{theorem}

 Theorems \ref{thm:avgsignature} and \ref{thm:avg4genus} and Inequality \eqref{eq:sigbound} together imply the following upper and lower bounds for the average 4-genus. For any $\varepsilon>0$ there is a crossing number $c_0$ such that for all $c\geq c_0$
\[ \sqrt{\frac{c}{2\pi}}-\varepsilon < \overline{g_4}(c)\leq \frac{9.75c}{\log c}.
\]
The authors expect that the $\varepsilon$ term can be removed from the above inequality.

Together with Theorem 1.1 in \cite{CohLow} by the first two authors on the average 3-genus $\overline{g_3}(c)$, which can be defined similarly, Theorem \ref{thm:avg4genus} immediately implies the following.

\begin{corollary}
\label{cor:limit}
The limit of the average 4-genus over the average 3-genus is
$$\lim_{{c\to\infty}} \frac{\overline{g_4}(c)}{\overline{g_3}(c)}=0.$$
\end{corollary}

Baader, Kjuchukova, Lewark, Misev, and Ray \cite{BKLMR} prove in their Theorem 1 a version of this result indexed by the number of crossings $n$ in the diagram and conjecture that it holds when indexed by the crossing number $c$, which Corollary \ref{cor:limit} achieves.

A less major theorem but perhaps one of the most surprising results of the paper is Theorem \ref{thm:binomials}, which states that over the set $T(2m+1)\cup T(2m+2)$, the number of knot diagrams with a given signature is a binomial coefficient $\displaystyle\binom{2m-1}{k}$.  Perhaps this gives more evidence towards some of the big theoretical questions posed earlier in this section.

This paper is organized as follows. Section \ref{sec:background} provides background on the construction of the set $T(c)$ of words representing 2-bridge knots with a given crossing number $c$ from work by the first two authors \cite{CohLow}. Section \ref{sec:signature} considers the number $s(c,\sigma)$ of words in $T(c)$ corresponding to a knot with signature $\sigma$.  Two useful recursions are given in Lemmas \ref{lem:signaturerecursion} and \ref{lem:signaturerecursion2}, yielding the data in Table \ref{tab:sig}.  Section \ref{sec:average} uses these results to give the proof of Theorem \ref{thm:avgsignature}, which relies on several technical lemmas that are given and proved afterwards in Subsection \ref{subsec:technical}. Section \ref{sec:4-genus} establishes Theorem \ref{thm:avg4genus} on 4-genus. Subsection \ref{subsec:newmodel} sets up a slightly different random model for $T(2m+1)\cup T(2m+2)$ and adapts useful results from \cite{BKLMR} to this setting to build a cobordism from the knot to a connected sum of knots.  Subsection \ref{subsec:random} introduces a random walk model for this connected sum.  Subsection \ref{subsec:analysis} summarizes an upper bound for 4-genus and then provides the analysis of this upper bound in order to prove the theorem.

\subsection{Acknowledgements}  
The third author was supported in part by a Discovery Grant from NSERC Canada. We are grateful to the Banff International Research Station for hosting the first and third authors at the Knot Theory Informed by Random Models and Experimental Data (24w5217) workshop and to the organizers of the workshop. Thanks also to Robert Adler and Matt Kahle for helpful conversations.

\section{Background}
\label{sec:background}

In this section, we provide background on 2-bridge knots with a particular focus on the set $T(c)$ whose construction is given in Definition \ref{def:tc}.

Ernst and Sumners computed the number of 2-bridge knots with chiral pairs \textit{not} counted separately.
\begin{theorem}[Ernst-Sumners \cite{ErnSum}, Theorem 5]
\label{thm:ernstsumners}
The number $|\mathcal{K}(c)|$ of 2-bridge knots with $c$ crossings where chiral pairs are \emph{not} counted separately is given by
\[
|\mathcal{K}(c)| = 
\begin{cases}
\frac{1}{3}(2^{c-3}+2^{\frac{c-4}{2}}) & \text{ for }4 \leq c\equiv 0 \text{ mod }4,\\
\frac{1}{3}(2^{c-3}+2^{\frac{c-3}{2}}) & \text{ for }5\leq c\equiv 1 \text{ mod }4, \\
\frac{1}{3}(2^{c-3}+2^{\frac{c-4}{2}}-1) & \text{ for }6 \leq c\equiv 2 \text{ mod }4, \text{ and}\\
\frac{1}{3}(2^{c-3}+2^{\frac{c-3}{2}}+1) & \text{ for }3\leq c\equiv 3 \text{ mod }4.
\end{cases}
\]
\end{theorem}

Instead of considering $\mathcal{K}(c)$ directly, we consider the set $T(c)$ of partially-double counted 2-bridge knots with $c$ crossings, where chiral pairs are not counted separately.

\begin{definition}
\label{def:tc}
\cite{CohLow, CohLowURSI}
Define the \emph{partially double-counted set $T(c)$ of $2$-bridge words with crossing number $c$} as follows. Each word in $T(c)$ is a word in the symbols $\{+,-\}$. If $c$ is odd, then a word $w$ is in $T(c)$ if and only if it is of the form
\[
(+)^{\varepsilon_1}(-)^{\varepsilon_2}(+)^{\varepsilon_3}(-)^{\varepsilon_4}\ldots(-)^{\varepsilon_{c-1}}(+)^{\varepsilon_c}, \]
where $\varepsilon_i\in\{1,2\}$ for $i\in\{1,\ldots,c\}$, $\varepsilon_1=\varepsilon_c=1$, and the length of the word $\ell=\sum_{i=1}^{c}\varepsilon_i \equiv 1$ mod $3$. Similarly, if $c$ is even, then a word $w$ is in $T(c)$ if and only if it is of the form
\[(+)^{\varepsilon_1}(-)^{\varepsilon_2}(+)^{\varepsilon_3}(-)^{\varepsilon_4}\ldots(+)^{\varepsilon_{c-1}}(-)^{\varepsilon_c},\]
where $\varepsilon_i\in\{1,2\}$ for $i\in\{1,\ldots,c\}$, $\varepsilon_1=\varepsilon_c=1$, and the length of the word $\ell=\sum_{i=1}^{c}\varepsilon_i \equiv 1$ mod $3$. 

A \textit{run} in a word $w\in T(c)$ is a subword of the form $(+)^{\varepsilon_i}$ or $(-)^{\varepsilon_i}$. By construction, each word in $T(c)$ has exactly $c$ runs.
\end{definition}

As in \cite{Coh:lower}, each word $w\in T(c)$ produces an alternating diagram $D_w$ consisting of two horizontal rows of crossings capped off with a plat closure. The runs  $(+)^1$ and $(-)^2$ of $w$ correspond to $\sigma_1$ crossings $\tikz[baseline=.6ex, scale = .4]{
\draw (0,0) -- (1,1);
\draw (0,1) -- (.3,.7);
\draw (.7,.3) -- (1,0);
}
~$ in the lower row, and the runs $(-)^1$ and $(+)^2$ correspond to $\sigma_2^{-1}$ crossings $\tikz[baseline=.6ex, scale = .4]{
\draw (0,0) -- (.3,.3);
\draw (.7,.7) -- (1,1);
\draw (0,1) -- (1,0);
}
~$ in the higher row. The $2$-bridge knot with diagram $D_w$ is denoted by $K_w$.

\begin{example}
\label{ex:word}
The word $+--++--++-$ in $T(6)$ corresponds to the braid word $\sigma_1^2\sigma_2^{-1}\sigma_1\sigma_2^{-2}$ and gives the alternating diagram shown in Figure \ref{fig:ex1}.
\end{example}

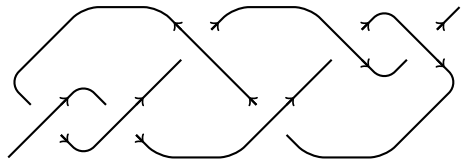
\begin{figure}[h!]
\begin{tikzpicture}[rounded corners = 2mm, thick]
    \draw (0,0) -- (1,1) -- (1.3,.7);
\draw (1.7,.3) -- (2,0) -- (3,0) -- (4.3,1.3);
\draw (4.7,1.7) -- (5,2) -- (6,1) -- (5,0) -- (4,0) -- (3.7,.3);
\draw (3.3,.7) -- (2,2) -- (1,2) -- (0,1) -- (.3,.7);
\draw (.7,.3) -- (1,0) -- (2.3,1.3);
\draw (2.7,1.7) -- (3,2) --(4,2) -- (5,1) -- (5.3,1.3);
\draw (5.7,1.7) -- (6,2);

\draw[->] (.7,.7)--(.8,.8);
\draw[->] (.7,.3) -- (.8,.2);
\draw[->] (1.7,.7) -- (1.8,.8);
\draw[->] (1.7,.3) -- (1.8,.2);
\draw[->] (2.3,1.7) -- (2.2,1.8);
\draw[->] (2.7,1.7) -- (2.8,1.8);
\draw[->] (3.3,.7) -- (3.2,.8);
\draw[->] (3.7,.7) -- (3.8,.8);
\draw[->] (4.7,1.7) -- (4.8,1.8);
\draw[->] (4.7,1.3) -- (4.8,1.2);
\draw[->] (5.7,1.7) -- (5.8,1.8);
\draw[->] (5.7,1.3) -- (5.8,1.2);

\end{tikzpicture}
\caption{The alternating diagram $D_w$ obtained from the word $+--++--++-$. Even though we do not draw it, the open strands on the left and right are understood to be connected.}
\label{fig:ex1}
\end{figure}

When $c$ is odd, a word $w\in T(c)$ is $\textit{palindromic}$ if $w$ and the reverse of $w$ are the same word, and when $c$ is even, $w$ is palindromic if $w$ is equal to the word obtained by reversing $w$ and exchanging all $+$'s and $-$'s. Let $T_p(c)$ be the subset of palindromic words in $T(c)$. The word $w$ in Example \ref{ex:word} is palindromic.

The next theorem describes the relationship between the set $\mathcal{K}(c)$ of 2-bridge knots and the set $T(c)$.  The statement below  appears as Theorem 2.5 in \cite{CohLow}. It follows from Lemma 2.20 in Cohen and Krishnan \cite{CoKr}, whose proof uses the work of Schubert \cite{Sch} and Koseleff and Pecker \cite{KosPec4}. A logically equivalent version of this result without the $T(c)$ notation can be found by combining Lemma 2.1, Assumption 2.2, and Remark 2.3 in Cohen \cite{Coh:lower}. 

\begin{theorem}[\cite{CohLow}]
\label{thm:list}
    Each knot in $\mathcal{K}(c)$ is represented by one or two words in $T(c)$. A knot is represented by exactly one word $w$ in ${T(c)}$ if and only if $w$ is palindromic.
\end{theorem}
Theorem \ref{thm:list} implies that $2|\mathcal{K}(c)|=|T(c)|+|T_p(c)|$.

The leftmost crossing in the alternating diagram $D_w$ is always oriented horizontally and to the right. This determines the orientations of all the other crossings. The first author \cite{Coh:lower} determined the orientations of all the crossings directly from the word $w\in T(c)$, as indicated in Table \ref{tab:orientations}.
\begin{table}[h]
\begin{tikzpicture}[scale=.8]

\draw (0,.25) node {Possible};
\draw (0,-.25) node {orientation(s)};
\draw (0,1.5) node {Position mod 3};
\draw (0,2.5) node {Run in $T(c)$};
\begin{scope}[scale=.5,baseline=0, xshift=5cm, yshift = -1cm]
    \draw[->] (0,0) -- (1,1);
    \draw (0,1) -- (0.3,0.7);
    \draw[->] (0.7,0.3) -- (1,0);
    \draw [->] (1,2) -- (0,2);
\end{scope}

\draw (2.75,1.5) node{1};
\draw (2.75,2.5) node{$+$};

\begin{scope}[scale=.5,baseline=0, xshift=9cm, yshift = -1cm]
    \draw[<-] (0,0) -- (1,1);
    \draw (0,1) -- (0.3,0.7);
    \draw[->] (0.7,0.3) -- (1,0);
    \draw [<-] (1,2) -- (0,2);

    \draw (2,1) node{or};

    \begin{scope}[xshift = 3cm]
    \draw[->] (0,0) -- (1,1);
    \draw[<-] (0,1) -- (0.3,0.7);
    \draw (0.7,0.3) -- (1,0);
    \draw [<-] (1,2) -- (0,2);
    \end{scope}
\end{scope}

\draw (5.5,1.5) node{2 or 3};
\draw (5.5,2.5) node{$+$};

\begin{scope}[xshift=5.5cm]
\begin{scope}[scale=.5,baseline=0, xshift=5cm, yshift = -1cm]
    \draw[->] (0,2) -- (1,1);
    \draw (0,1) -- (0.3,1.3);
    \draw[->] (0.7,1.7) -- (1,2);
    \draw[->] (1,0) -- (0,0);
\end{scope}

\draw (2.75,1.5) node{1};
\draw (2.75,2.5) node{$-$};
\end{scope}

\begin{scope}[xshift = 5.5cm]
\begin{scope}[scale=.5,baseline=0, xshift=9cm, yshift = -1cm]
    \draw[->] (0,2) -- (1,1);
    \draw[<-] (0,1) -- (0.3,1.3);
    \draw (0.7,1.7) -- (1,2);
    \draw[<-] (1,0) -- (0,0);

    \draw (2,1) node{or};

    \begin{scope}[xshift = 3cm]
    \draw[<-] (0,2) -- (1,1);
    \draw (0,1) -- (0.3,1.3);
    \draw[->] (0.7,1.7) -- (1,2);
    \draw[<-] (1,0) -- (0,0);
    \end{scope}
\end{scope}

\draw (5.5,1.5) node{2 or 3};
\draw (5.5,2.5) node{$-$};
\end{scope}
\draw (-2,-1) rectangle (12.75,3);
\draw (-2,2) -- (12.75,2);
\draw (-2,1) -- (12.75,1);
\draw (1.9,-1) -- (1.9,3);
\draw (3.75,-1) -- (3.75,3);
\draw (7.25,-1) -- (7.25,3);
\draw (9.25,-1) -- (9.25,3);

\begin{scope}[yshift = -4.5cm]
\draw (0,.25) node {Possible};
\draw (0,-.25) node {orientation(s)};
\draw (0,1.5) node {Positions mod 3};
\draw (0,2.5) node {Run in $T(c)$};
\begin{scope}[scale=.5,baseline=0, xshift=5cm, yshift = -1cm]
    \draw[->] (0,2) -- (1,1);
    \draw (0,1) -- (0.3,1.3);
    \draw[->] (0.7,1.7) -- (1,2);
    \draw[->] (1,0) -- (0,0);
\end{scope}

\draw (2.75,1.5) node{2-3};
\draw (2.75,2.5) node{$++$};

\begin{scope}[scale=.5,baseline=0, xshift=9cm, yshift = -1cm]
    \draw[->] (0,2) -- (1,1);
    \draw[<-] (0,1) -- (0.3,1.3);
    \draw (0.7,1.7) -- (1,2);
    \draw[<-] (1,0) -- (0,0);

    \draw (2,1) node{or};

    \begin{scope}[xshift = 3cm]
    \draw[<-] (0,2) -- (1,1);
    \draw (0,1) -- (0.3,1.3);
    \draw[->] (0.7,1.7) -- (1,2);
    \draw[<-] (1,0) -- (0,0);
    \end{scope}
\end{scope}

\draw (5.5,1.5) node{3-1 or 1-2};
\draw (5.5,2.5) node{$++$};

\begin{scope}[xshift=5.5cm]
\begin{scope}[scale=.5,baseline=0, xshift=5cm, yshift = -1cm]
    \draw[->] (0,0) -- (1,1);
    \draw (0,1) -- (0.3,0.7);
    \draw[->] (0.7,0.3) -- (1,0);
    \draw [->] (1,2) -- (0,2);
\end{scope}

\draw (2.75,1.5) node{ 2-3};
\draw (2.75,2.5) node{$--$};
\end{scope}

\begin{scope}[xshift = 5.5cm]
\begin{scope}[scale=.5,baseline=0, xshift=9cm, yshift = -1cm]
    \draw[<-] (0,0) -- (1,1);
    \draw (0,1) -- (0.3,0.7);
    \draw[->] (0.7,0.3) -- (1,0);
    \draw [<-] (1,2) -- (0,2);

    \draw (2,1) node{or};

    \begin{scope}[xshift = 3cm]
    \draw[->] (0,0) -- (1,1);
    \draw[<-] (0,1) -- (0.3,0.7);
    \draw (0.7,0.3) -- (1,0);
    \draw [<-] (1,2) -- (0,2);
    \end{scope}
\end{scope}

\draw (5.5,1.5) node{3-1 or 1-2};
\draw (5.5,2.5) node{$--$};
\end{scope}
\draw (-2,-1) rectangle (12.75,3);
\draw (-2,2) -- (12.75,2);
\draw (-2,1) -- (12.75,1);
\draw (1.9,-1) -- (1.9,3);
\draw (3.75,-1) -- (3.75,3);
\draw (7.25,-1) -- (7.25,3);
\draw (9.25,-1) -- (9.25,3);

\end{scope}

\end{tikzpicture}
\caption{ Orientations of crossings in the alternating diagram associated with a word in $T(c)$.}
\label{tab:orientations}
\end{table}

Our main results take advantage of the underlying combinatorial structure of the set $T(c)$. The first result we mention highlighting this combinatorial structure gives the size of $T(c)$.

\begin{proposition}
\cite{CohLow}
The number $t(c) = |T(c)| = \frac{2^{c-2} - (-1)^c}{3}$ of words in the set $T(c)$ is the Jacobsthal number $J(c-2)$ \cite{OEIS1045} and satisfies the same recursive formula $t(c)=t(c-1)+2t(c-2)$ and the same initial values $t(2)=J(0)=0$ and $t(3)=J(1)=1$ as the Jacobsthal sequence.
\end{proposition}

\begin{fact}
\label{fact:Jn}
The Jacobsthal numbers $J(n)$ satisfy the property $J(n)+J({n+1})=2^n$.
\end{fact}

In Section 3 of \cite{CohLow}, the first two authors found a useful recursive structure in the set $T(c)$; we will use this structure again to prove Lemma \ref{lem:signaturerecursion}.  Partition the set $T(c)$ into four subsets $T_i(c)$ for $1\leq i\leq 4$ according to the final three runs of the word as in Table \ref{tab:bijection}. So, for example, if $c$ is odd, then $T_1(c)$ is the subset of words in $T(c)$ whose final three runs are $+-+$.  The word $+--++--++-$ in Example \ref{ex:word} is in $T_2(6)$. Replacing the final three runs of a word $w$ with the runs in the fourth column of Table \ref{tab:bijection} yields a bijection between $T_i(c)$ and the subset in the fifth column of Table \ref{tab:bijection}.

\begin{table}[h]
    \begin{tabular}{|c|c|c|c|c|}
    \hline
    Parity of $c$ & $T_i(c)$ & Final 3 runs & Replacement & Replacement subset\\
    \hline
    \multirow{4}{*}{Odd} & $T_1(c)$ &  $+-+$ & $++-$ & $T_2(c-1)\cup T_3(c-1)$\\
    \cline{2-5}
     & $T_2(c)$ & $++--+$ & $+-$ & $T_1(c-1)\cup T_4(c-1)$\\
    \cline{2-5}
     & $T_3(c)$ & $+--+$ & $+$ & $T(c-2)$\\
    \cline{2-5}
     & $T_4(c)$ & $++-+$ & $+$ & $T(c-2)$\\
    \hline
    \multirow{4}{*}{Even} & $T_1(c)$ &  $-+-$ & $--+$ & $T_2(c-1)\cup T_3(c-1)$\\
    \cline{2-5}
     & $T_2(c)$ & $--++-$ & $-+$ & $T_1(c-1)\cup T_4(c-1)$\\
    \cline{2-5}
     & $T_3(c)$ & $-++-$ & $-$ & $T(c-2)$\\
    \cline{2-5}
     & $T_4(c)$ & $--+-$ & $-$ & $T(c-2)$\\
    \hline
    \end{tabular}
    \caption{The set $T_i(c)$ is defined to be all words in $T(c)$ whose final 3 runs are as indicated in the third column. Replacing those final 3 runs with the string in the fourth column defines a bijection between $T_i(c)$ and the set in the fifth column.}
    \label{tab:bijection}
\end{table}

\section{Signature}
\label{sec:signature}

In this section, we prove two recursive formulas for the number of $2$-bridge diagrams $D_w$ associated with words $w\in T(c)$ of a fixed crossing number and fixed signature. We also show that if one considers two consecutive crossing numbers at once, then the number of $2$-bridge diagrams $D_w$ associated with words $w$ in $T(2m+1)$ or $T(2m+2)$ of a fixed signature $\sigma=2k-2m+2$ is the binomial coefficient $\binom{2m-1}{k}$. These recursive and closed formulas are used in Section \ref{sec:average} to prove Theorem \ref{thm:avgsignature}, giving an asymptotic formula for the average of the absolute value of the signature of $2$-bridge knots with a fixed crossing number.

The following definition is the quantity of primary interest in this section.
\begin{definition}
    \label{def:s(c,sig)}
    Define $s(c,\sigma)$ to be the number of words $w$ in $T(c)$ whose associated knot $K_w$ has signature $\sigma$. 
\end{definition}

Traczyk \cite{Tra} obtained a formula for the signature of a non-split alternating link that only uses the following combinatorial data from an alternating diagram of the link. If every crossing $\crosspos$ in $D$ is replaced by its \textit{$A$-resolution} $\crosssmooth$, then one obtains a collection of simple closed curves called the \textit{all-$A$ resolution} of $D$. The number of curves in the all-$A$ resolution of $D$ is denoted by $s_A=s_A(D)$. An oriented crossing of the form $\crossposor$ is \textit{positive}, and the number of positive crossings in a diagram $D$ is denoted by $c_+=c_+(D)$. With these definitions, we state a version of Traczyk's result coming from Dasbach and Lowrance \cite{DasLow:sig}.
\begin{theorem}
\label{thm:Tra}
\cite{Tra}
Let $L$ be a non-split alternating link with alternating diagram $D$. 
The signature of $L$ is
\begin{equation}
\sigma(L)=s_A-c_+-1,
\end{equation}
where $s_A$ is the number of circles in the all-$A$-smoothing of $D$ and $c_+$ is the number of positive crossings of  $D$.
\end{theorem}

\begin{example}
\label{ex:basecases}
We look at examples of $T(c)$ for $3\leq c \leq 6$ to help the reader gain intuition and also to serve as base cases in proofs below, including Lemma \ref{lem:signaturerecursion2}, Proposition \ref{prop:sym}, and Theorem \ref{thm:binomials}.

Figure \ref{fig:t3sig} shows the alternating diagrams associated to the two words in $T(3)$ and $T(4)$ and their all-$A$ resolutions. 

The only word in $T(3)$ is $+--+$, which corresponds to the braid word $\sigma_1^3$. The diagram has no positive crossings and three components in its all-$A$ resolution. Therefore $\sigma=2$. The only word in $T(4)$ is $+-+-$, which corresponds to the braid word $\sigma_1 \sigma_2^{-1}\sigma_1\sigma_2^{-1}$. The two highlighted crossings are positive, and the diagram has three components in its all-$A$ resolution. Therefore $\sigma=0$. Thus $s(3,2)=1$ and $s(4,0)=1$, and we have that $s(3,\sigma)=s(4,\sigma)=0$ for all other values of $\sigma$.

Table \ref{tab:t56} shows the words in $T(5)$ and $T(6)$, the braid words of their associated alternating diagrams, and the combinatorial data needed to apply Traczyk's formula. The resulting values of $s(c,\sigma)$ are recorded in Table \ref{tab:sig} below.
\end{example}

\begin{figure}[h]
    \[\begin{tikzpicture}[thick]
    \begin{scope}[rounded corners = 3mm, xshift=1cm]
        \draw (2.7,.3) -- (3,0) -- (2,-1) -- (1,-1) -- (0,0) -- (1,1) -- (1.3,.7);
        \draw (0.7,.3) -- (1,0) -- (2,1) -- (2.3,.7);
        \draw (1.7,.3) -- (2,0) -- (3,1) -- (2,2) -- (1,2) -- (0,1) -- (0.3,0.7);
        \foreach \i in {0,1,2} {
            \draw[->] (\i+0.7,0.3) -- (\i+0.8,0.2);
            \draw[->] (\i+0.7,0.7) -- (\i+0.8,0.8);
            }
        \draw (1.5,-1.5) node{$c_+=0$};
    \end{scope}
    
    \begin{scope}[rounded corners = 3mm, xshift = 7cm]
        \draw (1.5,2) -- (1,2) -- (0,1) -- (.5,.5) -- (0,0) -- (1,-1) -- (2,-1) -- (3,0) -- (2.5,.5) -- (3,1) -- (2,2) -- (1.5,2);
        \draw (.75,.75) -- (1,1) -- (1.5,.5) -- (1,0) -- (.5,.5) -- (.75,.75);
        \draw[xshift = 1 cm] (.75,.75) -- (1,1) -- (1.5,.5) -- (1,0) -- (.5,.5) -- (.75,.75);
        \draw (1.5,-1.5) node{$s_A=3$};
    \end{scope}

    \begin{scope}[rounded corners = 3mm, yshift=-4cm]
         \fill[red!20!white] (2.5,.5) circle (.4cm);
        \fill[red!20!white] (3.5,1.5) circle (.4cm);
        \draw (3.7,1.7) -- (4,2) -- (5,0.5) -- (4,-1) -- (1,-1) -- (0,0) -- (1.3,1.3);
        \draw (1.7,1.7) -- (2,2) -- (3,2) -- (4,1) -- (3,0) -- (2.7,.3);
        \draw (2.3,.7) -- (1,2) -- (0,1) -- (.3,.7);
        \draw (.7,.3) -- (1,0) -- (2,0) -- (3.3,1.3);

        \draw[->] (.7,.3) -- (.8,.2);
        \draw[->] (.7,.7) -- (.8,.8);
        \draw[->] (1.3,1.7) -- (1.2,1.8);
        \draw[->] (1.7,1.7) -- (1.8,1.8);
        \draw[->] (2.3,.7) -- (2.2,.8);
        \draw[->] (2.7,.7) -- (2.8,.8);
        \draw[->] (3.7,1.3) -- (3.8,1.2);
        \draw[->] (3.7,1.7) -- (3.8,1.8);
       
        \draw (2.5,-1.5) node{$c_+=2$};

    \end{scope}
    
    \begin{scope}[rounded corners = 3mm, xshift = 6cm, yshift=-4cm]
        \draw (2,-1) -- (1,-1) -- (0,0) -- (.5,.5) -- (0,1) -- (1,2) -- (1.5,1.5) -- (2,2) -- (3,2) -- (3.5,1.5) -- (4,2) -- (5,0.5) -- (4,-1)-- (2,-1);
        \draw (.75,.75) -- (1.5,1.5) -- (2.5,.5) -- (2,0) -- (1,0) -- (.5,.5) -- (.75,.75);
        \draw (2.75,.75) -- (3.5,1.5) -- (4,1) -- (3,0) -- (2.5,.5) -- (2.75,.75);
        \draw (2.5,-1.5) node{$s_A=3$};
    \end{scope}

    \end{tikzpicture}
    \]
    \caption{On the left, the word $+--+$ in $T(3)$ corresponding to the alternating diagram with braid word $\sigma_1^3$ and the word $+-+-$ in $T(4)$ corresponding to the alternating diagram with braid word $\sigma_1\sigma_2^{-1}\sigma_1\sigma_2^{-1}$.  On the right, their all-$A$ resolutions.  For both we indicate the number $c_+$ of positive crossings (shaded) and the number of components $s_A$ in the all-$A$ resolution, used to compute signature $\sigma$.}
    \label{fig:t3sig}
\end{figure}
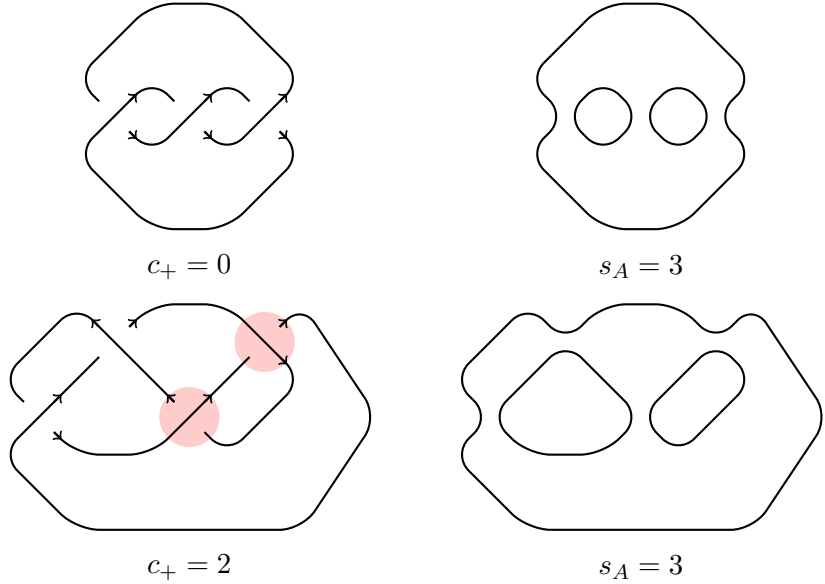

\begin{table}[h]
    \begin{tabular}{|c|c|c|c|c|c|}
    \hline
    $c$ & Word $w$ in $T(c)$ & Braid word & $c_+$ & $s_A$ & $\sigma$\\
    \hline\hline
    \multirow{3}{*}{5} & $+--+--+$ & $\sigma_1^5$  & 0 & 5 & 4\\
    \cline{2-6}
    & $+--++-+$ & $\sigma_1^2\sigma_2^{-2}\sigma_1$  & 0 & 3 & 2\\
    \cline{2-6}
    & $ +-++--+$ & $\sigma_1\sigma_2^{-2}\sigma_1^2$  & 0 & 3 & 2\\
    \hline
    \hline
    \multirow{5}{*}{6} & $+-+-++-$ & $\sigma_1\sigma_2^{-1}\sigma_1\sigma_2^{3}$  & 4 & 3 & -2\\
    \cline{2-6}
    & $+-+--+-$ & $\sigma_1\sigma_2^{-1}\sigma_1^4$  & 4 & 5 & 0\\
    \cline{2-6}
    & $+--++--++-$ & $\sigma_1^2\sigma_2^{-1}\sigma_1\sigma_2^{-2}$  & 3 & 4 & 0\\
    \cline{2-6}
    & $+-++-+-$ & $\sigma_1\sigma_2^{-3}\sigma_1\sigma_2^{-1}$  & 2 & 3 & 0\\
    \cline{2-6}
    & $+--+-+-$ & $\sigma_1^3\sigma_2^{-1}\sigma_1\sigma_2^{-1}$ &  2 & 5 & 2\\
    \hline
    \end{tabular}
    \caption{Each row contains a word in $T(5)$ or $T(6)$, the braid word of the corresponding alternating diagram, the number of positive crossings $c_+$, the number of components $s_A$ in the all-$A$ resolution, and the signature $\sigma$.}
    \label{tab:t56}
\end{table}

In the following lemma, we find our first recursive formula for $s(c,\sigma)$. Its proof follows from the bijections summarized in Table \ref{tab:bijection} and the changes to the signature induced by the replacement summarized in Figure \ref{fig:s(c,sigma)}. See Table \ref{tab:sig} for more values of $s(c,\sigma)$ when $c$ is small, as determined by this recursion.

\begin{remark}
We observe that this table seems to be identical to the sequence given in \cite{OEIS013580}, except that our $\sigma=2$ column has an additional $+1$ for odd crossing numbers.
\end{remark}

\begin{lemma}
\label{lem:signaturerecursion}
If $c\geq 5$, the number $s(c,\sigma)$ of words in $T(c)$ corresponding to a knot with signature $\sigma$ satisfies the recurrence relation
\begin{equation*}
s(c,\sigma)=s(c-1,\sigma+(-1)^c2)+s(c-2,\sigma+(-1)^c2)+s(c-2,\sigma).
\end{equation*}
\end{lemma}

\begin{table}[h!]
\begin{tabular}{c|cccc|c|c|c|ccccc}
$c\backslash\sigma$ & -10 & -8 & -6 & -4 & -2 & 0 & 2 & 4 & 6 & 8 & 10 & 12 \\
\hline
3 &   &   &   &  &  &  & 1 &  &  &  &  &  \\
4 &   &   &   &  &  & 1 &  &  &  &  &  &  \\
5 &   &   &   &  &  &  & 2 & 1 &  &  &  &  \\
6 &  &  &  &   & 1 & 3 & 1 &  &  &  &  &  \\
7 &  &  &  &  &  & 1 & 5 & 4 & 1 &  &  &  \\
8 &  &  &  & 1 & 5 & 9 & 5 & 1 &  &  &  &  \\
9 &  &  &  &  & 1 & 6 & 15 & 14 & 6 & 1 &  &  \\
10 &  &  & 1 & 7 & 20 & 29 & 20 & 7 & 1 &  &  &  \\
11 &  &  &  & 1 & 8 & 27 & 50 & 49 & 27 & 8 & 1 &  \\
12 &  & 1 & 9 & 35 & 76 & 99 & 76 & 35 & 9 & 1 &  &  \\
13 &  &  & 1 & 10 & 44 & 111 & 176 & 175 & 111 & 44 & 10 & 1 \\
14 & 1 & 11 & 54 & 155 & 286 & 351 & 286 & 155 & 54 & 11 & 1 & 
\end{tabular}
\caption{Numbers $s(c,\sigma)$ of knots in $T(c)$ with $c$ crossings and signature $\sigma$, as generated from Example \ref{ex:basecases} and Lemma \ref{lem:signaturerecursion}.}
\label{tab:sig}
\end{table}

\begin{proof}
Define $S_i(c,\sigma)$ to be the subset of $T_i(c)$ defined in Section \ref{sec:background} consisting of words where $\sigma(K_w)=\sigma$, and define $s_i(c,\sigma)$ to be the number of words in $S_i(c,\sigma)$. 

 Figure \ref{fig:s(c,sigma)} shows that the bijections described in Table \ref{tab:bijection} respect signature in the following sense. Suppose two words have the same crossing number, signature, and are in the same case according to Table \ref{tab:bijection}: that is, both words are in $S_i(c,\sigma)$ for some $i$, $c$, and $\sigma$. Then the images of the two words under the bijections described in Table \ref{tab:bijection} have the same crossing number and signature, and both are elements of the set in the last column of Table \ref{tab:bijection}. Therefore, the bijections from Table \ref{tab:bijection} restrict to bijections when the signature $\sigma$ is fixed, that is, on the sets $S_i(c,\sigma)$.

If $c$ is  odd, Table \ref{tab:bijection} and Figure \ref{fig:s(c,sigma)} imply 
\begin{align*}
    s_1(c,\sigma) = & \; s_2(c-1,\sigma-2) + s_3(c-1,\sigma-2),\\
    s_2(c,\sigma) = & \; s_1(c-1,\sigma-2) + s_4(c-1,\sigma-2),\\
    s_3(c,\sigma) = & \; s(c-2,\sigma-2),~\text{and}\\
    s_4(c,\sigma) = & \; s(c-2,\sigma).
\end{align*}
Adding each of these equations together yields
\[s(c,\sigma) = s(c-1,\sigma-2) + s(c-2,\sigma-2)+s(c-2,\sigma),\]
proving the result when $c$ is  odd.

If $c$ is even, Table \ref{tab:bijection} and Figure \ref{fig:s(c,sigma)} imply
\begin{align*}
    s_1(c,\sigma) = & \; s_2(c-1,\sigma+2) + s_3(c-1,\sigma+2),\\
    s_2(c,\sigma) = & \; s_1(c-1,\sigma+2) + s_4(c-1,\sigma+2),\\
    s_3(c,\sigma) = & \; s(c-2,\sigma+2),~\text{and}\\
    s_4(c,\sigma) = & \; s(c-2,\sigma).
\end{align*}
Adding each of these equations together yields
\[s(c,\sigma) = s(c-1,\sigma+2) + s(c-2,\sigma+2)+s(c-2,\sigma),\]
proving the result when $c$ is even.
\end{proof}

\begin{figure}[!h]
\[\begin{tikzpicture}[thick]



\draw (-.25,1.5) node{\small{Final}};
\draw (-.25,1.1) node{\small{Runs}};

\draw (1.9,1.5) node{\small{Alternating}};
\draw (1.9,1.1) node{\small{Diagram}};

\draw (4,1.5) node{\small{All-$A$}};
\draw (4,1.1) node{\small{Smoothing}};

\draw(-1,-.75) rectangle (5.5,.75);
\draw(0,0) node{${\scriptstyle +-+}$};
\begin{scope}[scale=.5, rounded corners = 1mm, xshift = 2.5cm, yshift = -1cm]

\fill[white!80!red] (.5,.5) circle (.5cm);

\draw (0,0) -- (1.3, 1.3);
\draw (0,1) -- (.3,.7);
\draw (.7,.3) -- (1,0) -- (2,0) -- (3,1) -- (2,2) -- (1.7,1.7);
\draw (0,2) -- (1,2) -- (2.3,.7);
\draw (2.7,.3) -- (3,0);

\draw[thick,->] (.5, .5) -- (.1,.1);
\draw[thick,->] (.7,.3) -- (.9,.1);
\draw[->] (2.5, .5) -- (2.9,.9);
\draw[->] (2.7,.3) -- (2.9,.1);
\draw[->] (1.5, 1.5) -- (1.9,1.1);
\draw[->] (1.3,1.3) -- (1.1,1.1);

\end{scope}

\begin{scope}[scale=.5, rounded corners = 1mm, xshift = 6.75cm, yshift = -1cm]

\draw[-] (0,1) -- (.4,.5) -- (0,0);
\draw[-] (0,2) -- (1,2) -- (1.5,1.6) -- (2,2) -- (3,1) -- (2.6,.5) -- (3,0);

\draw[-] (1.5,1.4) -- (.6,.5) -- (1,0) -- (2,0) -- (2.4,.5) -- (2,1) -- cycle;

\end{scope}

\draw[->] (5.5,0) -- (7,0);
\draw (6.25,.5) node{\small{Case $1_o$}};
\draw (6.25,-.5) node{\small{$\Delta\sigma {=} {-}2$}};

\begin{scope}[xshift = 8 cm]
\draw (-.25,1.5) node{\small{Final}};
\draw (-.25,1.1) node{\small{Runs}};

\draw (1.9,1.5) node{\small{Alternating}};
\draw (1.9,1.1) node{\small{Diagram}};

\draw (4,1.5) node{\small{All-$A$}};
\draw (4,1.1) node{\small{Smoothing}};
\draw(-1,-.75) rectangle (5.5,.75);
\draw(0,0) node{${\scriptstyle ++-}$};
\begin{scope}[scale=.5, rounded corners = 1mm, xshift = 3.5cm, yshift = -1cm]

\fill[white!80!red] (.5,1.5) circle (.5cm);
\fill[white!80!red] (1.5,1.5) circle (.5cm);
\draw (0,0) -- (1,0) -- (2,1) -- (1,2) -- (.6,1.6);
\draw (.4,1.4) -- (0,1);
\draw (0,2) -- (1,1) -- (1.4,1.4);
\draw (1.6,1.6) -- (2,2);
\draw[thick,->] (0.7,1.7) -- (.9,1.9);
\draw[thick,->] (.7,1.3) -- (.9,1.1);
\draw[thick,->] (1.7,1.7) -- (1.9,1.9);
\draw[thick,->] (1.7, 1.3) -- (1.9,1.1);

\end{scope}

\begin{scope}[scale=.5, rounded corners = 1mm, xshift = 7.5cm, yshift = -1cm]

\draw[-] (0,2) -- (.5,1.6) -- (1,2) -- (1.5,1.6) -- (2,2);
\draw[-] (0,1) -- (.5, 1.4) -- (1,1) -- (1.5,1.4) -- (2,1) -- (1,0) -- (0,0);

\end{scope}
\end{scope}

\begin{scope}[yshift = -2cm]

\draw(-1,-.75) rectangle (5.5,.75);
\draw(0,0) node{${\scriptstyle ++--+}$};
\begin{scope}[scale=.5, rounded corners = 1mm, xshift = 2.5cm, yshift = -1cm]

\draw (0,0) -- (1,0) -- (2,1) -- (2.3,.7);
\draw (2.7,.3) -- (3,0);
\draw (0,1) -- (.3,1.3);
\draw (0,2) -- (1.3,.7);
\draw (1.7,.3) -- (2,0) -- (3,1) -- (2,2) -- (1,2) -- (.7,1.7);
\draw[->] (.3,1.3) -- (.1,1.1);
\draw[->] (.5,1.5) -- (.9,1.1);
\draw[->] (1.5,.5) -- (1.9,.9);
\draw[->] (1.7,.3) -- (1.9,.1);
\draw[->] (2.5,.5) -- (2.9,.9);
\draw[->] (2.7,.3) -- (2.9,.1);

\end{scope}

\begin{scope}[scale=.5, rounded corners = 1mm, xshift = 5.5cm, yshift = -1cm]

    \foreach \i/\j in {1/1} {
        \draw (\i,\j+1) -- (\i+.5,\j+.6) -- (\i+1,\j+1);
        \draw (\i,\j)   -- (\i+.5,\j+.4) -- (\i+1,\j);
    }
    \foreach \i/\j in {2/0,3/0} {
        \draw (\i,\j)   -- (\i+.4,\j+.5) -- (\i,\j+1);
        \draw (\i+1,\j) -- (\i+.6,\j+.5) -- (\i+1,\j+1);
    }
\draw (1,0) -- (2,0);
\draw (2,2) -- (3,2) -- (4,1);
\end{scope}

\draw[->] (5.5,0) -- (7,0);
\draw (6.25,.5) node{\small{Case $2_o$}};
\draw (6.25,-.5) node{\small{$\Delta\sigma {=} {-}2$}};

\end{scope}

\begin{scope}[xshift = 8cm, yshift=-2cm]

\draw(-1,-.75) rectangle (5.5,.75);

\draw(0,0) node{${\scriptstyle +-}$};
\begin{scope}[scale=.5, rounded corners = 1mm, xshift = 3.5cm, yshift = -1cm]

\fill[white!80!red] (.5,.5) circle (.5cm);
\fill[white!80!red] (1.5,1.5) circle (.5cm);
\draw (0,0) -- (1.3,1.3);
\draw (1.7,1.7)--(2,2);
\draw (0,1) -- (0.3,0.7);
\draw (0.7,0.3) -- (1,0) -- (2,1) -- (1,2) -- (0,2);

\draw[thick,->] (0.5,0.5) -- (0.9, 0.9);
\draw[thick,->] (0.3,0.7) -- (0.1,0.9);
\draw[thick,->] (1.5,1.5) -- (1.9,1.1);
\draw[thick,->] (1.7, 1.7) -- (1.9, 1.9);

\end{scope}

\begin{scope}[scale=.5, rounded corners = 1mm, xshift = 6.5cm, yshift = -1cm]

    \foreach \i/\j in {2/1} {
        \draw (\i,\j+1) -- (\i+.5,\j+.6) -- (\i+1,\j+1);
        \draw (\i,\j)   -- (\i+.5,\j+.4) -- (\i+1,\j);
    }
    \foreach \i/\j in {1/0} {
        \draw (\i,\j)   -- (\i+.4,\j+.5) -- (\i,\j+1);
        \draw (\i+1,\j) -- (\i+.6,\j+.5) -- (\i+1,\j+1);
    }
\draw[-] (2,0)--(3,1);
\draw[-] (1,2)--(2,2);
\end{scope}
\end{scope}

\begin{scope}[yshift = -4cm]

\draw(-1,-.75) rectangle (5.5,.75);
\draw(0,0) node{${\scriptstyle +--+}$};
\begin{scope}[scale=.5, rounded corners = 1mm, xshift = 2.5cm, yshift = -1cm]

\draw (0,0) -- (1,1) -- (1.3,.7);
\draw (0,1) -- (.3,.7);
\draw (.7,.3) -- (1,0) -- (2,1) -- (2.3,.7);
\draw (1.7,.3) -- (2,0) -- (3,1) -- (2,2) -- (0,2);
\draw (2.7,.3) -- (3,0);
\draw[->] (0.5,0.5) -- (0.9,0.9);
\draw[->] (0.7,0.3) -- (0.9, 0.1);
\draw[->] (1.5,0.5) -- (1.9,0.9);
\draw[->] (1.7,0.3) -- (1.9, 0.1);
\draw[->] (2.5,0.5) -- (2.9,0.9);
\draw[->] (2.7,0.3) -- (2.9, 0.1);
\end{scope}

\begin{scope}[scale=.5, rounded corners = 1mm, xshift = 6.5cm, yshift = -1cm]
    \foreach \i/\j in {0/0,1/0,2/0} {
        \draw (\i,\j)   -- (\i+.4,\j+.5) -- (\i,\j+1);
        \draw (\i+1,\j) -- (\i+.6,\j+.5) -- (\i+1,\j+1);
    }
\draw (0,2) -- (2,2) -- (3,1);
\end{scope}

\draw[->] (5.5,0) -- (7,0);
\draw (6.25,.5) node{\small{Case $3_o$}};
\draw (6.25,-.5) node{\small{$\Delta\sigma {=} {-}2$}};

\end{scope}

\begin{scope}[xshift = 8cm, yshift = -4cm]

\draw(-1,-.75) rectangle (5.5,.75);

\draw(0,0) node{${\scriptstyle +}$};
\begin{scope}[scale=.5, rounded corners = 1mm, xshift = 3.5cm, yshift = -1cm]

\draw (0,1) -- (.3,.7);
\draw (.7,.3) -- (1,0);
\draw (0,0) -- (1,1) -- (0,2);

\draw[->] (.5,.5) -- (.9,.9);
\draw[->] (.7,.3) -- (.9,.1);

\end{scope}

\begin{scope}[scale=.5, rounded corners = 1mm, xshift = 7.5cm, yshift = -1cm]
    \foreach \i/\j in {0/0} {
        \draw (\i,\j)   -- (\i+.4,\j+.5) -- (\i,\j+1);
        \draw (\i+1,\j) -- (\i+.6,\j+.5) -- (\i+1,\j+1);
    }
\draw (0,2) -- (1,1);
\end{scope}
\end{scope}

\begin{scope}[yshift = -6cm]

\draw(-1,-.75) rectangle (5.5,.75);
\draw(0,0) node{${\scriptstyle ++-+}$};
\begin{scope}[scale=.5, rounded corners = 1mm, xshift = 2.5cm, yshift = -1cm]
\draw (0,0) -- (2,0) -- (3,1) -- (2,2) -- (1.7,1.7);
\draw (1.3,1.3) -- (1,1) -- (0,2);
\draw (0,1) -- (.3,1.3);
\draw (.7,1.7) -- (1,2) -- (2.3,.7);
\draw (2.7,0.3) -- (3,0);

\draw[->] (0.5, 1.5) -- (0.1, 1.9);
\draw[->] (0.7,1.7) -- (0.9,1.9);
\draw[->] (1.5, 1.5) -- (1.9,1.1);
\draw[->] (1.3,1.3) -- (1.1, 1.1);
\draw[->] (2.5,0.5) -- (2.9,0.9);
\draw[->] (2.7,0.3) -- (2.9, 0.1);
\end{scope}

\begin{scope}[scale=.5, rounded corners = 1mm, xshift = 6.5cm, yshift = -1cm]
    \foreach \i/\j in {0/1,1/1} {
        \draw (\i,\j+1) -- (\i+.5,\j+.6) -- (\i+1,\j+1);
        \draw (\i,\j)   -- (\i+.5,\j+.4) -- (\i+1,\j);
    }
    \foreach \i/\j in {2/0} {
        \draw (\i,\j)   -- (\i+.4,\j+.5) -- (\i,\j+1);
        \draw (\i+1,\j) -- (\i+.6,\j+.5) -- (\i+1,\j+1);
    }
\draw (0,0) -- (2,0);
\draw (2,2) -- (3,1);
\end{scope}
\draw[->] (5.5,0) -- (7,0);
\draw (6.25,.5) node{\small{Case $4_o$}};
\draw (6.25,-.5) node{\small{$\Delta\sigma {=} 0$}};

\end{scope}

\begin{scope}[xshift = 8cm, yshift = -6cm]
\draw(-1,-.75) rectangle (5.5,.75);

\draw(0,0) node{${\scriptstyle +}$};
\begin{scope}[scale=.5, rounded corners = 1mm, xshift = 3.5cm, yshift = -1cm]

\draw (0,1) -- (.3,.7);
\draw (.7,.3) -- (1,0);
\draw (0,0) -- (1,1) -- (0,2);

\draw[->] (.5,.5) -- (.9,.9);
\draw[->] (.7,.3) -- (.9,.1);

\end{scope}

\begin{scope}[scale=.5, rounded corners = 1mm, xshift = 7.5cm, yshift = -1cm]
    \foreach \i/\j in {0/0} {
        \draw (\i,\j)   -- (\i+.4,\j+.5) -- (\i,\j+1);
        \draw (\i+1,\j) -- (\i+.6,\j+.5) -- (\i+1,\j+1);
    }
\draw (0,2) -- (1,1);
\end{scope}
\end{scope}



\begin{scope}[yshift = -8cm]

\draw(-1,-.75) rectangle (5.5,.75);
\draw(0,0) node{${\scriptstyle -+-}$};
\begin{scope}[scale=.5, rounded corners = 1mm, xshift = 2.5cm, yshift = -1cm]
\fill[white!80!red] (1.5,.5) circle (.5cm);
\fill[white!80!red] (2.5,1.5) circle (.5cm);
\draw (0,0) -- (1,0) -- (2,1) -- (2.3,1.3);
\draw (2.7,1.7) -- (3,2);
\draw (0,2) -- (1,1) -- (1.3,.7);
\draw (1.7,.3) -- (2,0) -- (3,1) -- (2,2) -- (1,2) -- (.7,1.7);
\draw (0,1) -- (.3,1.3);

\draw[->] (.3,1.7) -- (.1,1.9);
\draw[->] (.7,1.7) -- (.9,1.9);
\draw[thick, ->] (1.7,.7) -- (1.9,.9);
\draw[thick, ->] (1.3,.7) -- (1.1,.9);
\draw[thick,->] (2.7,1.7) -- (2.9,1.9);
\draw[thick,->] (2.7,1.3) -- (2.9,1.1);

\end{scope}

\begin{scope}[scale=.5, rounded corners = 1mm, xshift = 6.5cm, yshift = -1cm]
    \foreach \i/\j in {0/1,2/1} {
        \draw (\i,\j+1) -- (\i+.5,\j+.6) -- (\i+1,\j+1);
        \draw (\i,\j)   -- (\i+.5,\j+.4) -- (\i+1,\j);
    }
    \foreach \i/\j in {1/0} {
        \draw (\i,\j)   -- (\i+.4,\j+.5) -- (\i,\j+1);
        \draw (\i+1,\j) -- (\i+.6,\j+.5) -- (\i+1,\j+1);
    }
\draw (0,0)--(1,0);
\draw (1,2)--(2,2);
\draw (2,0)--(3,1);
\end{scope}

\draw[->] (5.5,0) -- (7,0);
\draw (6.25,.5) node{\small{Case $1_e$}};
\draw (6.25,-.5) node{\small{$\Delta\sigma {=} 2$}};

\begin{scope}[xshift = 8 cm]
\draw(-1,-.75) rectangle (5.5,.75);
\draw(0,0) node{${\scriptstyle -{}-+}$};
\begin{scope}[scale=.5, rounded corners = 1mm, xshift = 3.5cm, yshift = -1cm]
\draw (0,2) -- (1,2) -- (2,1) -- (1,0) -- (.7,.3);
\draw (.3,.7) -- (0,1);
\draw (0,0) -- (1,1) -- (1.3,.7);
\draw (1.7,.3) -- (2,0);
\draw[->] (0.7,.3) -- (.9,.1);
\draw[->] (.7,.7) -- (.9,.9);
\draw[->] (1.7,.3) -- (1.9,.1);
\draw[->] (1.7, .7) -- (1.9,.9);

\end{scope}

\begin{scope}[scale=.5, rounded corners = 1mm, xshift = 7.5cm, yshift = -1cm]

    \foreach \i/\j in {0/0,1/0} {
        \draw (\i,\j)   -- (\i+.4,\j+.5) -- (\i,\j+1);
        \draw (\i+1,\j) -- (\i+.6,\j+.5) -- (\i+1,\j+1);
    }
\draw (0,2) -- (1,2) -- (2,1);
\end{scope}
\end{scope}
\end{scope}

\begin{scope}[yshift = -10cm]

\draw(-1,-.75) rectangle (5.5,.75);
\draw(0,0) node{${\scriptstyle -{}-++-}$};
\begin{scope}[scale=.5, rounded corners = 1mm, xshift = 2.5cm, yshift = -1cm]

\fill[white!80!red] (.5,.5) circle (.5cm);
\fill[white!80!red] (1.5,1.5) circle (.5cm);
\fill[white!80!red] (2.5,1.5) circle (.5cm);
\draw (1,2) -- (2,1) -- (2.3,1.3);
\draw (2.7,1.7) -- (3,2);
\draw (0,1) -- (.3,.7);
\draw (0,0) -- (1.3,1.3);
\draw (1.7,1.7) -- (2,2) -- (3,1) -- (2,0) -- (1,0) -- (.7,.3);
\draw[thick,->] (.3,.7) -- (.1,.9);
\draw[thick,->] (.5,.5) -- (.9,.9);
\draw[thick,->] (1.5,1.5) -- (1.9,1.1);
\draw[thick,->] (1.7,1.7) -- (1.9,1.9);
\draw[thick,->] (2.5,1.5) -- (2.9,1.1);
\draw[thick,->] (2.7,1.7) -- (2.9,1.9);
\draw (0,2) -- (1,2);
\end{scope}

\begin{scope}[scale=.5, rounded corners = 1mm, xshift = 5.5cm, yshift = -1cm]

    \foreach \i/\j in {2/1,3/1} {
        \draw (\i,\j+1) -- (\i+.5,\j+.6) -- (\i+1,\j+1);
        \draw (\i,\j)   -- (\i+.5,\j+.4) -- (\i+1,\j);
    }
    \foreach \i/\j in {1/0} {
        \draw (\i,\j)   -- (\i+.4,\j+.5) -- (\i,\j+1);
        \draw (\i+1,\j) -- (\i+.6,\j+.5) -- (\i+1,\j+1);
    }
\draw (1,2) -- (2,2);
\draw (2,0) -- (3,0) -- (4,1);
\end{scope}

\draw[->] (5.5,0) -- (7,0);
\draw (6.25,.5) node{\small{Case $2_e$}};
\draw (6.25,-.5) node{\small{$\Delta\sigma {=} 2$}};
\end{scope}

\begin{scope}[xshift = 8cm, yshift=-10cm]

\draw(-1,-.75) rectangle (5.5,.75);

\draw(0,0) node{${\scriptstyle -+}$};
\begin{scope}[scale=.5, rounded corners = 1mm, xshift = 3.5cm, yshift = -1cm]

\draw (0,2) -- (1.3,.7);
\draw (1.7,.3)--(2,0);
\draw (0,1) -- (0.3,1.3);
\draw (0.7,1.7) -- (1,2) -- (2,1) -- (1,0) -- (0,0);

\draw[->] (0.3,1.3) -- (0.1, 1.1);
\draw[->] (0.7,1.3) -- (0.9,1.1);
\draw[->] (1.7,0.7) -- (1.9,0.9);
\draw[->] (1.7, 0.3) -- (1.9, 0.1);

\end{scope}

\begin{scope}[scale=.5, rounded corners = 1mm, xshift = 6.5cm, yshift = -1cm]

    \foreach \i/\j in {1/1} {
        \draw (\i,\j+1) -- (\i+.5,\j+.6) -- (\i+1,\j+1);
        \draw (\i,\j)   -- (\i+.5,\j+.4) -- (\i+1,\j);
    }
    \foreach \i/\j in {2/0} {
        \draw (\i,\j)   -- (\i+.4,\j+.5) -- (\i,\j+1);
        \draw (\i+1,\j) -- (\i+.6,\j+.5) -- (\i+1,\j+1);
    }
\draw[-] (2,2) -- (3,1);
\draw[-] (1,0)--(2,0);
\end{scope}
\end{scope}

\begin{scope}[yshift = -12cm]

\draw(-1,-.75) rectangle (5.5,.75);
\draw(0,0) node{${\scriptstyle -++-}$};
\begin{scope}[scale=.5, rounded corners = 1mm, xshift = 2.5cm, yshift = -1cm]

\fill[white!80!red] (.5,1.5) circle (.5cm);
\fill[white!80!red] (1.5,1.5) circle (.5cm);
\fill[white!80!red] (2.5,1.5) circle (.5cm);
\draw (0,2) -- (1,1) -- (1.3,1.3);
\draw (0,1) -- (.3,1.3);
\draw (.7,1.7) -- (1,2) -- (2,1) -- (2.3,1.3);
\draw (1.7,1.7) -- (2,2) -- (3,1) -- (2,0) -- (0,0);
\draw (2.7,1.7) -- (3,2);
\draw[thick,->] (0.5,1.5) -- (0.9,1.1);
\draw[thick,->] (0.7,1.7) -- (0.9,1.9);
\draw[thick,->] (1.5,1.5) -- (1.9,1.1);
\draw[thick,->] (1.7,1.7) -- (1.9,1.9);
\draw[thick,->] (2.5,1.5) -- (2.9,1.1);
\draw[thick,->] (2.7,1.7) -- (2.9,1.9);
\end{scope}

\begin{scope}[scale=.5, rounded corners = 1mm, xshift = 6.5cm, yshift = -1cm]
    \foreach \i/\j in {0/1,1/1,2/1} {
        \draw (\i,\j+1) -- (\i+.5,\j+.6) -- (\i+1,\j+1);
        \draw (\i,\j)   -- (\i+.5,\j+.4) -- (\i+1,\j);
    }
\draw (0,0) -- (2,0) -- (3,1);
\end{scope}

\draw[->] (5.5,0) -- (7,0);
\draw (6.25,.5) node{\small{Case $3_e$}};
\draw (6.25,-.5) node{\small{$\Delta\sigma {=} 2$}};

\end{scope}

\begin{scope}[xshift = 8cm, yshift = -12cm]

\draw(-1,-.75) rectangle (5.5,.75);

\draw(0,0) node{${\scriptstyle -}$};
\begin{scope}[scale=.5, rounded corners = 1mm, xshift = 3.5cm, yshift = -1cm]

\fill[white!80!red] (.5,1.5) circle (.5cm);
\draw (0,1) -- (.3,1.3);
\draw (.7,1.7) -- (1,2);
\draw (0,2) -- (1,1) -- (0,0);

\draw[thick,->] (.5,1.5) -- (.9,1.1);
\draw[thick,->] (.7,1.7) -- (.9,1.9);

\end{scope}

\begin{scope}[scale=.5, rounded corners = 1mm, xshift = 7.5cm, yshift = -1cm]
    \foreach \i/\j in {0/1} {
        \draw (\i,\j+1) -- (\i+.5,\j+.6) -- (\i+1,\j+1);
        \draw (\i,\j)   -- (\i+.5,\j+.4) -- (\i+1,\j);
    }
\draw (0,0) -- (1,1);
\end{scope}
\end{scope}

\begin{scope}[yshift = -14cm]

\draw(-1,-.75) rectangle (5.5,.75);
\draw(0,0) node{${\scriptstyle --+-}$};
\begin{scope}[scale=.5, rounded corners = 1mm, xshift = 2.5cm, yshift = -1cm]
\fill[white!80!red] (.5,.5) circle (.5cm);
\fill[white!80!red] (1.5,.5) circle (.5cm);
\fill[white!80!red] (2.5,1.5) circle (.5cm);
\draw (0,2) -- (2,2) -- (3,1) -- (2,0) -- (1.7,.3);
\draw (1.3,.7) -- (1,1) -- (0,0);
\draw (0,1) -- (.3,.7);
\draw (.7,.3) -- (1,0) -- (2.3,1.3);
\draw (2.7,1.7) -- (3,2);

\draw[thick,->] (0.5,.5) -- (0.1,.1);
\draw[thick,->] (0.7,.3) -- (0.9,.1);
\draw[thick,->] (1.5,.5) -- (1.9,.9);
\draw[thick,->] (1.3,.7) -- (1.1,.9);
\draw[thick,->] (2.5,1.5) --(2.9,1.1);
\draw[thick,->] (2.7,1.7) --(2.9,1.9);
\end{scope}

\begin{scope}[scale=.5, rounded corners = 1mm, xshift = 6.5cm, yshift = -1cm]
    \foreach \i/\j in {2/1} {
        \draw (\i,\j+1) -- (\i+.5,\j+.6) -- (\i+1,\j+1);
        \draw (\i,\j)   -- (\i+.5,\j+.4) -- (\i+1,\j);
    }
    \foreach \i/\j in {0/0,1/0} {
        \draw (\i,\j)   -- (\i+.4,\j+.5) -- (\i,\j+1);
        \draw (\i+1,\j) -- (\i+.6,\j+.5) -- (\i+1,\j+1);
    }
\draw (0,2) -- (2,2);
\draw (2,0) -- (3,1);
\end{scope}
\draw[->] (5.5,0) -- (7,0);
\draw (6.25,.5) node{\small{Case $4_e$}};
\draw (6.25,-.5) node{\small{$\Delta\sigma {=} 0$}};

\end{scope}

\begin{scope}[xshift = 8cm, yshift = -14cm]
\draw(-1,-.75) rectangle (5.5,.75);

\draw(0,0) node{${\scriptstyle -}$};
\begin{scope}[scale=.5, rounded corners = 1mm, xshift = 3.5cm, yshift = -1cm]

\fill[white!80!red] (.5,1.5) circle (.5cm);
\draw (0,1) -- (.3,1.3);
\draw (.7,1.7) -- (1,2);
\draw (0,2) -- (1,1) -- (0,0);

\draw[thick,->] (.5,1.5) -- (.9,1.1);
\draw[thick,->] (.7,1.7) -- (.9,1.9);

\end{scope}

\begin{scope}[scale=.5, rounded corners = 1mm, xshift = 7.5cm, yshift = -1cm]
    \foreach \i/\j in {0/1} {
        \draw (\i,\j+1) -- (\i+.5,\j+.6) -- (\i+1,\j+1);
        \draw (\i,\j)   -- (\i+.5,\j+.4) -- (\i+1,\j);
    }
\draw (0,0) -- (1,1);
\end{scope}
\end{scope}

\end{tikzpicture}\]

\caption{The final runs, alternating diagrams, and partial $A$-resolutions correspond to the cases in Lemma \ref{lem:signaturerecursion}, with positive crossings shaded. In each case $ \Delta \sigma = \sigma(K_2)-\sigma(K_1)$ where $K_1$ and $K_2$ are the knots on the left and right, respectively. The knots on the left correspond to words in $T(c)$ while the knots on the right correspond to knots in $T(c-1)\cup T(c-2)$, and the correspondence between left and right is the same as in Table \ref{tab:bijection}.}
\label{fig:s(c,sigma)}
\end{figure}

The recursion in Lemma \ref{lem:signaturerecursion} implies another recursion that we will use to prove Theorem \ref{thm:avgsignature}.

\newpage
\begin{lemma}
\label{lem:signaturerecursion2}
If $c\geq 4$, the number $s(c, \sigma)$ of words in $T(c)$ corresponding to a knot with signature $\sigma$ satisfies a recurrence relation similar to that for binomial coefficients:
\begin{equation}
\label{eq:recursion2}
s(c,\sigma)=
\begin{cases}
s(c-1,\sigma-2)+s(c-1,\sigma-4) & \text{when $c$ is odd and $\sigma\neq 2$,}\\
s(c-1,\sigma-2)+s(c-1,\sigma-4)+1 & \text{when $c$ is odd and $\sigma= 2$,}\\
s(c-1,\sigma+2)+s(c-1,\sigma+4) & \text{when $c$ is even and $\sigma\neq -2$, and}\\
s(c-1,\sigma+2)+s(c-1,\sigma+4)-1 & \text{when $c$ is even and $\sigma= -2$.}\\
\end{cases}
\end{equation}
\end{lemma}

\begin{proof}
We proceed by induction with the base cases handled by Example \ref{ex:basecases}:  the row corresponding to $c=4$ in Table \ref{tab:sig} does indeed follow from the row corresponding to $c=3$ by the second two cases in Equation \eqref{eq:recursion2}, and the row corresponding to $c=5$ there follows from the row corresponding to $c=4$ by the first two.

By Lemma \ref{lem:signaturerecursion}, for odd $c$ and $\sigma\neq 2$ we have that $s(c,\sigma)=s(c-1,\sigma-2)+s(c-2,\sigma-2)+s(c-2,\sigma)$.  So in order to show that $s(c,\sigma)=s(c-1,\sigma-2)+s(c-1,\sigma-4)$, we have only to show that $s(c-2,\sigma-2)+s(c-2,\sigma)=s(c-1,\sigma-4)$, which holds by the induction hypothesis. For the $\sigma = 2$ case, we still have the relation $s(c,\sigma)=s(c-1,\sigma-2)+s(c-2,\sigma-2)+s(c-2,\sigma)$. Specializing to $\sigma = 2$, we have $s(c,2)=s(c-1,0)+s(c-2,0)+s(c-2,2)$. By the inductive hypothesis, $s(c-2,0)+s(c-2,2) = s(c-1,-2)+1$. Thus $s(c,2)=s(c-1,0)+s(c-1,-2)+1$.

The case for even $c$ follows similarly.
\end{proof}

The rows in Table \ref{tab:sig} corresponding to even crossing numbers are symmetric about the $\sigma=0$ entry, and the rows corresponding to odd crossing numbers are mostly symmetric about the $\sigma=2$ and $\sigma=4$ entries, except $s(c,2)=s(c,4)+1$ in those rows. Lemma \ref{lem:signaturerecursion2} implies that these symmetries hold in general.

\begin{proposition}
\label{prop:sym}
    The number $s(c,\sigma)$ of words in $T(c)$ satisfies the following symmetries:
    \begin{alignat*}{3}
        s(c,\sigma) = & \; s(c,-\sigma) & \quad &\text{if $c$ is even,}\\
    s(c,2) = & \; s(c,4)+1 & \quad & \text{if $c$ is odd and $\sigma$ is 2 or 4, and}\\
        s(c,\sigma)  = & \;   s(c,6-\sigma)& \quad & \text{if $c$ is odd and $\sigma\geq 6$~\text{or}~$\sigma\leq0$.}    
    \end{alignat*}
\end{proposition}

\begin{proof}
    Example \ref{ex:basecases} shows that the proposition is true when $c=3$ or $4$. We proceed by induction, with cases when $c$ is even or odd.
    
    Suppose $c$ is even. When $\sigma=2$,
    \[s(c,2) =  s(c-1,4)+s(c-1,6) =  (s(c-1,2)-1)+s(c-1,0) = s(c,-2).\]
    If $\sigma>2$, then we have
    \[s(c,\sigma) =  s(c-1,\sigma+2)+s(c-1,\sigma+4) =   s(c-1,4-\sigma)+s(c-1,2-\sigma) =   s(c,-\sigma),\]
    where in the previous two equations, the first and third equalities follow from Lemma \ref{lem:signaturerecursion2} and the second equality holds by the inductive hypothesis when $c$ is odd. 

    Suppose $c$ is odd. When $\sigma=2$,
   \[s(c,2)  =  s(c-1,0) + (s(c-1,-2) + 1)  =  s(c-1,0) + s(c-1,2) + 1  =  s(c,4) + 1.\]    
    If $\sigma \geq 6$, then we have
   \[
    s(c,\sigma)  =  s(c-1,\sigma-2) + s(c-1,\sigma-4)  =  s(c-1,2-\sigma) + s(c-1,4-\sigma)  =  s(c,6-\sigma),\]   
     where in the previous two equations, the first and third equalities follow from Lemma \ref{lem:signaturerecursion2} and the second equality holds by the inductive hypothesis when $c$ is even. 
\end{proof}

If we add consecutive rows in Table \ref{tab:sig} corresponding to $c=2m+1$ and $c=2m+2$, then we obtain the $n=2m-1$ row of Pascal's triangle. The following theorem shows this pattern continues.

\begin{theorem}
\label{thm:binomials}
Let $T_m=T(2m+1)\cup T(2m+2)$. Then there are exactly 
$\binom{2m-1}{k}=\binom{2m-1}{m-1+\frac{\sigma}{2}}$ words representing knots with signature $\sigma=2k-2m+2$. In other words, 
\[s(2m+1,\sigma) + s(2m+2,\sigma) = \binom{2m-1}{k}.\]
\end{theorem}

\begin{proof}
Let $s_{m,k}=s(2m+1,2k-2m+2)+s(2m+2,2k-2m+2)$.  We show $s_{m,k}=\binom{2m-1}{k}$.

The base case where $m=1$, and so $c=3$ and $4$, is satisfied by Example \ref{ex:basecases}.  Suppose that $s_{m-1,k}=\binom{2m-3}{k}$ for $0\leq k \leq 2m-3$ or 0 otherwise.

We first use Lemma \ref{lem:signaturerecursion} and then reorganize indices to utilize the $s_{m,k}$ notation: 
\begin{align}
\begin{split}
\label{eq:sigrec1}
s(2m+1,2k-2m+2) = & \; s(2m,2k-2m)+s(2m-1,2k-2m)\\
& \; +s(2m-1,2k-2m+2)  \\
 = &\; s(2m,2(k-2)-2(m-1)+2)\\
& \; + s(2m-1,2(k-2)-2(m-1)+2)  \\
& \; +s(2m-1,2k-2m+2) \\
 = & \; s_{m-1,k-2}+s(2m-1,2k-2m+2),
\end{split}
\end{align}
since $2k-2m=2(k-2)-2(m-1)+2$.  Similarly, and also by applying Equation \eqref{eq:sigrec1}:
\begin{align}
\begin{split}
\label{eq:sigrec2}
s(2m+2,2k-2m+2)  = & \; s(2m+1,2k-2m+4)+s(2m,2k-2m+4)\\
& \; +s(2m,2k-2m+2)  \\
 = &\; s(2m+1,2(k+1)-2m+2)+s(2m,2k-2(m-1)+2)\\
& \; +s(2m,2k-2m+2)  \\
 = &\; s_{m-1,k-1} + s(2m-1,2(k+1)-2m+2)  \\
& \; +s(2m,2k-2(m-1)+2)+s(2m,2k-2m+2) \\
 = &\; s_{m-1,k-1} + s_{m-1,k} + s(2m,2k-2m+2)
\end{split}
\end{align}
since $2(k+1)-2m+2=2k-2m+4=2k-2(m-1)+2$.

Putting Equations \eqref{eq:sigrec1} and \eqref{eq:sigrec2} 
together, we have
\begin{align*}
s_{m,k} & = s(2m+1,2k-2m+2)+s(2m+2,2k-2m+2)\\
& = s_{m-1,k-2}+s(2m-1,2k-2m+2)\\
& \; + s_{m-1,k-1} + s_{m-1,k} + s(2m,2k-2m+2)\\
& = s_{m-1,k-2} + s_{m-1,k-1} + s_{m-1,k}\\
& \; + s(2m-1,2(k-1)-2(m-1)+2) + s(2m,2(k-1)-2(m-1)+2)\\
& = s_{m-1,k-2} + 2s_{m-1,k-1} + s_{m-1,k},
\end{align*}
since $2k-2m+2=2(k-1)-2(m-1)+2$.

By the inductive hypothesis, we have
\begin{align*}
s_{m,k} & = \left[\binom{2m-3}{k-2} + \binom{2m-3}{k-1}\right]+\left[\binom{2m-3}{k-1} + \binom{2m-3}{k}\right]\\
& = \binom{2m-2}{k-1} + \binom{2m-2}{k} \; = \; \binom{2m-1}{k},
\end{align*}
completing the proof.
\end{proof}

After the first version of this paper was posted as a preprint, the first and third authors together with students \cite{BCDFMST} used the results in this section to show that the signature distribution approaches a normal distribution.

\section{Average signature}
\label{sec:average}

In this section, we prove Theorem \ref{thm:avgsignature}, which states that the average absolute value of the signature of $2$-bridge knots with crossing number $c$ satisfies
\[\lim_{c\to\infty}\left(\overline{|\sigma|}(c)-\sqrt{\frac{2c}{\pi}}\right)=0.\]
 First we prove Theorem \ref{thm:avgsignature} assuming several technical lemmas. In Subsection \ref{subsec:technical}, we state and prove the technical lemmas needed for the proof.

Define an equivalence relation $\approx$ on functions $f_1$ and $f_2$ by $f_1(c)\approx f_2(c)$ if $\lim_{c\to\infty}f_1(c)-f_2(c)=0$. Then Theorem \ref{thm:avgsignature} says that $\overline{|\sigma|}(c)\approx \sqrt{\frac{2c}{\pi}}$. 

\begin{definition}
\label{def:tot}
    The \textit{total absolute signature} $\tot(c)$ from $T(c)$ is defined as 
    \[\tot(c) = \sum_{w\in T(c)} |\sigma(K_w)| =  \sum_{i=-\infty}^\infty |2i| \; s(c,2i),\]
    and the \textit{total absolute palindromic signature} $\tot_p(c)$ from $T_p(c)$ is defined as 
    \[\tot_p(c) = \sum_{w\in T_p(c)} |\sigma(K_w)|.\]
\end{definition}

With this notation, $\overline{|\sigma|}(c)$ can be computed as
\[\overline{|\sigma|}(c) = \frac{\tot(c) + \tot_p(c)}{2|\mathcal{K}(c)|}.\]
The following proposition shows that the contribution $\tot_p(c)$ to the average signature coming from palindromic words does not change the asymptotic behavior of $\overline{|\sigma|}(c)$.
\begin{proposition}
\label{prop:ignorepalindromes}
    The average $\overline{|\sigma|}(c)$ of the absolute value of the signature of $2$-bridge knots with crossing number $c$ satisfies
    \[\overline{|\sigma|}(c)\approx \frac{\tot(c)}{2|\mathcal{K}(c)|}.\]
\end{proposition}
\begin{proof}
    Lemma \ref{lem:signaturerecursion2}  implies that if $K$ is a $2$-bridge knot with crossing number $c$, then $0\leq |\sigma(K)|\leq c-1$. The first two authors \cite{CohLow} proved the number of palindromic words $t_p(c)$ satisfies
    \[t_p(c)=\frac{1}{3}\left(2^{\left\lfloor\frac{c-1}{2}\right\rfloor}-(-1)^{\left\lfloor\frac{c-1}{2}\right\rfloor}\right),\]
    and therefore
    \[0\leq \tot_p(c)\leq \frac{1}{3}(c-1)\left(2^{\left\lfloor\frac{c-1}{2}\right\rfloor}-(-1)^{\left\lfloor\frac{c-1}{2}\right\rfloor}\right).\]
    Therefore Theorem \ref{thm:ernstsumners} implies that $\displaystyle\lim_{c\to\infty}\frac{\tot_p(c)}{2|\mathcal{K}(c)|}=0$, and thus
    \[\overline{|\sigma|}(c) = \frac{\tot(c)+\tot_p(c)}{2|\mathcal{K}(c)|}\approx \frac{\tot(c)}{2|\mathcal{K}(c)|},\]
    as desired.
\end{proof}
Proposition \ref{prop:ignorepalindromes} can be modified to show that palindromes do not contribute to the asymptotic behavior of the average value of many different invariants over the set of $2$-bridge knots. For example, see the work on genus \cite{CohLow, CohLowURSI}, braid index \cite{LowURSI:braids, SuzTra:braids}, and  crosscap number \cite{CohKinLowShaVan}.

Section \ref{sec:signature} gives various ways of counting the number of 2-bridge knots with a fixed signature and a fixed crossing number (or a pair of fixed crossing numbers). Lemmas \ref{lem:signaturerecursion} and \ref{lem:signaturerecursion2} and Theorem \ref{thm:binomials} can be used to estimate $\tot(c)$, leading to a proof of Theorem \ref{thm:avgsignature}. Although the proof of Theorem \ref{thm:avgsignature} relies on a few technical lemmas, we present it now, and use the remainder of this section to state and prove the various technical lemmas.

\begin{proof}[Proof of Theorem \ref{thm:avgsignature}]
Proposition \ref{prop:ignorepalindromes} states that $\overline{|\sigma|}(c)\approx \frac{\tot(c)}{2|\mathcal{K}(c)|}$. Theorem \ref{thm:binomials} leads to the following formula, found in Lemma \ref{lem:binomialsum}, for the 2-row total $\tot_2(m)$ of the absolute values of the signatures of $2$-bridge knots coming from $T_m=T(2m+1)\cup T(2m+2)$:
\[\tot_2(m) = \sum_{w\in T_m} |\sigma(K_w)| = m\binom{2m}{m}.\]
Lemma \ref{lem:signaturerecursion2} is used to find a relationship between $\tot(2m+1)$ and $\tot(2m+2)$, which is carefully stated in Lemma \ref{lem:totrecursive}. This relationship along with the fact that $\tot_2(m)=\tot(2m+1)+\tot(2m+2)$ implies the following inequality for $j=1$ and $2$:
\begin{equation}
\label{eq:totinequality}
   jm\binom{2m}{m}  - \varepsilon(2m+j)\leq 3\tot(2m+j)\leq jm\binom{2m}{m}+\varepsilon(2m+j)
    \end{equation}
where $c=2m+j$ and $\displaystyle\lim_{c\to\infty}\frac{\varepsilon(c)}{2|\mathcal{K}(c)|}=0$. Lemma \ref{lem:totinequality} gives a formula for $\varepsilon(c)$ and a proof of Inequality \eqref{eq:totinequality}.  Therefore, Inequality \eqref{eq:totinequality} implies that
\begin{equation}
    \label{eq:sigav1}
    \overline{|\sigma|}(2m+j) \approx \frac{\tot(c)}{2|\mathcal{K}(c)|} \approx \frac{jm \binom{2m}{m}}{6|\mathcal{K}(c)|}
\end{equation}
for $j=1$ and $2$.

Multiplying the inequality in Lemma \ref{lem:Wallis} by $\frac{jm}{6|\mathcal{K}(c)|}$ where $j=1$ or $2$ yields
\begin{equation}
\label{eq:UseWallis}
\frac{jm}{6|\mathcal{K}(c)|} \cdot \frac{4^m}{\sqrt{\pi m}} \left(1-\frac{1}{4m}\right) <  \frac{jm \binom{2m}{m}}{6|\mathcal{K}(c)|} < \frac{jm}{6|\mathcal{K}(c)|} \cdot \frac{4^m}{\sqrt{\pi m}}
\end{equation}
The Ernst-Sumners formula for $|\mathcal{K}(c)|$ in Theorem \ref{thm:ernstsumners} implies that
\[2^{c-2}\leq 6|\mathcal{K}(c)| \leq 2^{c-2}+2^{\left\lfloor\frac{c-1}{2}\right\rfloor}+2.\]
Therefore
\begin{equation}
\label{eq:UseES}
\frac{1}{2^{2m-1}+2^m+2} \leq \frac{j}{6|\mathcal{K}(c)|} \leq \frac{1}{2^{2m-1}},    
\end{equation}
where $j=1$ or $2$ and $c=2m+j$. Hence by Inequalities \eqref{eq:UseWallis} and \eqref{eq:UseES}
\[\frac{4^m}{2^{2m-1}+2^m+2}\sqrt{\frac{m}{\pi}}\left(1-\frac{1}{4m}\right) < \frac{jm \binom{2m}{m}}{6|\mathcal{K}(c)|} < 2\sqrt{\frac{m}{\pi}}. \]

Since
\begin{align*}
   &\; \lim_{m\to\infty} \left(2\sqrt{\frac{m}{\pi}} - \frac{4^m}{2^{2m-1}+2^m+2}\sqrt{\frac{m}{\pi}}\left(1-\frac{1}{4m}\right) \right)\\
   = &\; \lim_{m\to\infty}  \frac{4^{m}+8m2^m+16m}{ 4m(2^{2m-1}+2^m+2)} \sqrt{\frac{m}{\pi}}\\
    = &\; \lim_{m\to\infty}  \frac{2^{2m}+8m2^m+16m}{4\sqrt{m\pi}(2^{2m-1}+2^m+2)}  =0,
\end{align*}
Equation \eqref{eq:sigav1} implies
\[\overline{|\sigma|}(2m+j)\approx \frac{jm \binom{2m}{m}}{6|\mathcal{K}(c)|} \approx 2\sqrt{\frac{m}{\pi}}. \]
Since $c=2m+j$ for $j=1$ or $2$, we have
\[\overline{|\sigma|}(c)=\overline{|\sigma|}(2m+j) \approx 2\sqrt{\frac{m}{\pi}} = \sqrt{\frac{2(c-j)}{\pi}}\approx \sqrt{\frac{2c}{\pi}},\]
proving the result.
\end{proof}

\subsection{Technical lemmas needed for Theorem \ref{thm:avgsignature}} 
\label{subsec:technical}

In this subsection, we prove the lemmas used in the proof of Theorem \ref{thm:avgsignature}.

Theorem \ref{thm:binomials} allows us to exactly compute the sum of the absolute values of the signatures of knots coming from $T_m=T(2m+1)\cup T(2m+2)$.
\begin{lemma}
\label{lem:binomialsum}
    Let $\tot_2(m)$ be the 2-row total of the absolute values of the signatures of $2$-bridge knots coming from $T_m=T(2m+1)\cup T(2m+2)$, that is, let 
    \[\tot_2(m) =\tot(2m+1)+\tot(2m+2) = \sum_{w\in T_m} |\sigma(K_w)|.\]
    Then
    \[ \tot_2(m) = m\binom{2m}{m}.\]
\end{lemma}

\begin{proof}
Theorem \ref{thm:binomials} implies that
\[\tot_2(m) =  \sum_{i=1}^m 2i \binom{2m-1}{m-1+i} + \sum_{i=1}^{m-1} 2i \binom{2m-1}{m-1-i}.\]
Since $\binom{2m-1}{j} = \binom{2m-1}{2m-1-j}$, it follows that $\binom{2m-1}{m-1+i} = \binom{2m-1}{m-i}$.  Then 

\begin{align*}
\tot_{2}(m) = & \; \sum_{i=1}^m 2i \binom{2m-1}{m-i} + \sum_{i=1}^{m-1} 2i \binom{2m-1}{m-1-i} \\
= & \; 2m\binom{2m-1}{0}+2\sum_{i=1}^{m-1} i \left[\binom{2m-1}{m-i} + \binom{2m-1}{m-1-i}\right] = 2m+2\sum_{i=1}^{m-1} i \binom{2m}{m-i}.
\end{align*}

Setting $j=m-i$ yields
\begin{align*}
\tot_2(m)= & \; 2m+2\sum_{j=1}^{m-1} (m-j) \binom{2m}{j}\\
= & \; 2m\binom{2m}{0}+2m\sum_{j=1}^{m-1} \binom{2m}{j}-2\sum_{j=1}^{m-1} j \binom{2m}{j}\\
= & \; 2m\sum_{j=0}^{m-1} \binom{2m}{j}-2\sum_{j=1}^{m-1} j \binom{2m}{j}.
\end{align*}

Since $j\binom{2m}{j} = (2m)\binom{2m-1}{j-1}$, it follows that
\begin{align*}
\tot_2(m) = & \; 2m\sum_{j=0}^{m-1} \binom{2m}{j}-4\sum_{j=1}^{m-1} m \binom{2m-1}{j-1} \\
= & \; 2m\sum_{j=0}^{m-1} \binom{2m}{j}-4m\sum_{k=0}^{m-2}  \binom{2m-1}{k}.
\end{align*}

Since $\sum_{j=0}^{m-1}\binom{2m}{j}$ is the sum of all terms before the central term in an even row of Pascal's triangle, it follows that
\[\sum_{ j=0}^{m-1}\binom{2m}{j} = \frac{1}{2}\left(2^{2m}-\binom{2m}{m}\right).\]
Also, since $\sum_{k=0}^{m-1}\binom{2m-1}{k}$ is the sum of the first half of the entries in an odd row of Pascal's triangle, it follows that
\[ \sum_{k=0}^{m-1}\binom{2m-1}{k} = 2^{2m-2},\]
and so
\[ \sum_{k=0}^{m-2}\binom{2m-1}{k} = 2^{2m-2}-\binom{2m-1}{m-1}.\]
Therefore
\begin{align*}
\tot_2(m)= & \; 2m\cdot \frac{1}{2}\left[ 2^{2m}-\binom{2m}{m} \right]-4m\left[ 2^{2m-2}-\binom{2m-1}{m-1} \right]\\
= & \; m2^{2m}-m\binom{2m}{m} -m2^{2m}+4m\binom{2m-1}{m-1} = m\binom{2m}{m}\\
\end{align*}
since $4m\binom{2m-1}{m-1} = 2m\binom{2m}{m}$.
\end{proof}

Lemma \ref{lem:binomialsum} would allow us to compute the average absolute value of the signature of 2-bridge knots with crossing numbers $c=2m+1$ or $2m+2$. In order to compute this average for a single crossing number instead of two crossing numbers at once, we prove the following lemma relating $\tot(2m+2)$ with $\tot(2m+1)$.
\begin{lemma}
    \label{lem:totrecursive}
    If $c$ is even, then $\tot(c) = 2\tot(c-1)-2s(c-1,2)-6s(c-1,4) -2$.
\end{lemma}

\begin{proof}
The total absolute signature $\tot(c)$ from Definition \ref{def:tot} can be expressed as $\tot(c)=\tot_+(c) + \tot_-(c)$, where
$\tot_\pm(c) =  \sum_{i=1}^\infty 2i\; s(c,\pm 2i).$

    Lemma \ref{lem:signaturerecursion2} implies that if $c$ is even and $\sigma \neq -2$, then $s(c,\sigma) = s(c-1, \sigma +2) + s(c-1, \sigma + 4)$. Thus
\begin{align}
\label{eq:totpos}
\tot_+(c) = & \; \sum_{i=1}^\infty 2i\; s(c,2i) = \sum_{i=1}^\infty 2i \left[ s(c-1,2i+2) + s(c-1,2i+4) \right] \nonumber \\
= & \; \sum_{i=1}^\infty 2i\; s(c-1,2i+2) + \sum_{i=1}^\infty 2i\; s(c-1,2i+4).
\end{align}
    Substituting $j=i+1$ in the first summation of Equation \eqref{eq:totpos} yields
\begin{align*}
\sum_{i=1}^\infty 2i\; s(c-1,2i+2) = & \; \sum_{j=2}^\infty 2(j-1) s(c-1,2j) 
= \sum_{j=2}^\infty 2j\; s(c-1,2j) -2 \sum_{j=2}^\infty  s(c-1,2j)\\
= & \; \tot_+(c-1) -2\sum_{j=1}^\infty s(c-1,2j).
\end{align*}

 Substituting $\ell=i+2$ in the second summation of Equation \eqref{eq:totpos} yields
\begin{align*}
\sum_{i=1}^\infty 2i \; s(c-1,2i+4) = & \; \sum_{\ell=3}^\infty 2(\ell-2) s(c-1,2 \ell) 
=  \sum_{\ell=3}^\infty 2\ell \; s(c-1,2 \ell) -4 \sum_{\ell=3}^\infty  s(c-1,2 \ell) \\
= & \; \tot_+(c-1) - 2s(c-1,2) - 4s(c-1,4) - 4 \sum_{\ell=3}^\infty  s(c-1,2 \ell)\\
= & \; \tot_+(c-1) +2s(c-1,2) - 4\sum_{\ell=1}^\infty s(c-1,2\ell).
\end{align*}
Therefore 
\[\tot_+(c) = 2\tot_+(c-1) +2s(c-1,2)-6\sum_{i=1}^\infty s(c-1,2i).\]
We now find similar expressions for $\tot_-(c)$.
  Lemma \ref{lem:signaturerecursion2} implies that $s(c, -2) = s(c-1,0) + s(c-1,2) -1$, and thus 
\begin{equation}
\label{eq:totneg}
\tot_-(c)=\sum_{i=1}^\infty 2i \; s(c,-2i) = \sum_{i=1}^\infty 2i\; s(c-1,-2i+ 2) + \sum_{i=1}^\infty 2i\; s(c-1,-2i+ 4) - 2.
\end{equation}

Substituting $j=i-1$ in the first summation of Equation \eqref{eq:totneg} yields
\begin{align*}
\sum_{i=1}^\infty 2i\; s(c-1,-2i+ 2) = & \sum_{j=0}^\infty 2(j+1) s(c-1,-2j)\\
= & \; \sum_{j=0}^\infty 2j \; s(c-1,-2j) + 2\sum_{j=0}^\infty s(c-1,-2j)\\
= & \; \tot_-(c-1) + 2 \sum_{j=0}^\infty s(c-1,-2j).
\end{align*}

Substituting $\ell=i-2$ in the second summation of Equation \eqref{eq:totneg} yields 
\begin{align*}
\sum_{i=1}^\infty 2i \; s(c-1,-2i+ 4) = & \sum_{\ell=-1}^\infty 2(\ell + 2) s(c-1,-2\ell)\\
= & \; \sum_{\ell = -1}^\infty 2\ell \; s(c-1,-2\ell) + 4 \sum_{\ell = -1}^\infty  s(c-1,-2\ell)\\
= & \; -2s(c-1,2) + \tot_-(c-1) + 4s(c-1,2) +4 \sum_{\ell=0}^\infty s(c-1,-2\ell)\\
= & \; \tot_-(c-1) + 2s(c-1,2) +  4 \sum_{\ell=0}^\infty s(c-1,-2\ell).
\end{align*}
Therefore
\[\tot_-(c) = 2\tot_-(c-1) + 6 \sum_{i=0}^\infty s(c-1,-2i) + 2s(c-1,2) - 2.\]

Since $c-1$ is odd, by
Proposition \ref{prop:sym} we have $s(c-1, 6+2i) = s(c-1,-2i)$ for all $i\geq 0$, and thus $\sum_{i=3}^\infty s(c-1,2i) = \sum_{i=0}^\infty s(c-1,-2i).$ Hence the total signature $\tot(c)$ can be expressed as

\begin{align*}
\tot(c) = & \tot_+(c) + \tot_-(c)\\
= & \;2\tot_+(c-1) +2s(c-1,2)-6\sum_{i=1}^\infty s(c-1,2i) \\
& \;+ 2\tot_-(c-1) + 6 \sum_{i=0}^\infty s(c-1,-2i) + 2s(c-1,2) - 2 \\
= & \; 2\tot(c-1) -2s(c-1,2) - 6s(c-1,4) - 2,
\end{align*}
completing the proof. 
\end{proof}

Lemma \ref{lem:totrecursive} yields the inequalities in the following lemma for $\tot(2m+1)$ and $\tot(2m+2)$.
\begin{lemma}
    \label{lem:totinequality}
    The sum $\tot(c)$ of the absolute value of the signatures of knots coming from $T(c)$ satisfies the following inequality for $j=1$ and $2$:
    \[ jm\binom{2m}{m} - \varepsilon(2m+j) \leq 3\tot(2m+j)\leq jm\binom{2m}{m}+\varepsilon(2m+j),\]
 where $\varepsilon(2m+j) = 2\binom{2m-1}{m}+6\binom{2m-1}{m+1}+2$. Moreover, if $c=2m+1$ or $c=2m+2$, then
        \[\lim_{m\to\infty} \frac{\varepsilon(c)}{2|\mathcal{K}(c)|}=\lim_{m\to\infty} \frac{1}{2|\mathcal{K}(c)|} \left(2\binom{2m-1}{m}+6\binom{2m-1}{m+1}+2 \right) = 0.\]
\end{lemma}
\begin{proof}
By Lemmas \ref{lem:binomialsum} and \ref{lem:totrecursive}, 
\begin{align*}
    m\binom{2m}{m} = & \; \tot_2(m)\\
    = & \; \tot(2m+1) + \tot(2m+2)\\
    = & \; 3\tot(2m+1) -2s(2m+1,2) - 6s(2m+1,4) - 2.
\end{align*}
Theorem \ref{thm:binomials} implies that $s(2m+1,2)\leq \binom{2m-1}{m}$ and $s(2m+1,4)\leq \binom{2m-1}{m+1}$. Thus
\begin{align*}
    m\binom{2m}{m} \leq & \;  3\tot(2m+1)\\
    = & \; m\binom{2m}{m} + 2s(2m+1,2)+6s(2m+1,4)+2\\
    \leq & \; m\binom{2m}{m} +2 \binom{2m-1}{m} + 6\binom{2m-1}{m+1}+2\\
    = & \; m\binom{2m}{m} + \varepsilon(2m+1)
\end{align*}
proving the inequality in the statement of the lemma when $j=1$.

Again by Lemmas \ref{lem:binomialsum} and \ref{lem:totrecursive}, 
\begin{align*}
    2m\binom{2m}{m} = & \; 2\tot_2(m)\\
    = & \; 2\tot(2m+1) + 2\tot(2m+2)\\
    = & \; 3\tot(2m+2) + 2s(2m+1,2) + 6s(2m+1,4) + 2.
\end{align*}
Thus
\begin{align*}
2m\binom{2m}{m} - \varepsilon(2m+2) = & \; 2m\binom{2m}{m} -2 \binom{2m-1}{m} - 6\binom{2m-1}{m+1}-2\\
 \leq  & \;  2m\binom{2m}{m} - 2s(2m+1,2)- 6s(2m+1,4)-2\\
= & \; 3\tot(2m+2)\\
\leq & \; 2m\binom{2m}{m},
\end{align*}
proving the inequality in the statement of the lemma when $j=2$.

To show that 
        \[\lim_{m\to\infty} \frac{1}{2|\mathcal{K}(c)|} \left(2\binom{2m-1}{m}+6\binom{2m-1}{m+1}+2 \right) = 0,\]
we first observe that if $c=2m+1$, then
\[   \frac{2\binom{2m-1}{m}}{2^{c-2}} = \frac{\binom{2m}{m}}{2^{2m-1}}.\]
Lemma \ref{lem:Wallis} below implies that
\[\frac{2}{\sqrt{\pi m}}\left(1-\frac{1}{4m}\right) < \frac{\binom{2m}{m}}{2^{2m-1}} < \frac{2}{\sqrt{\pi m}}. \]
The squeeze theorem implies that $\frac{\binom{2m}{m}}{2^{2m-1}}=\frac{2\binom{2m-1}{m}}{2^{c-2}}$ goes to zero as $c$ goes to infinity, and thus so does $\frac{2\binom{2m-1}{m}}{2|\mathcal{K}(c)|}$. The argument when $c=2m+2$ is similar. Since $\binom{2m-1}{m+1} < \binom{2m-1}{m}$, it follows that  $\frac{6\binom{2m-1}{m+1}}{2|\mathcal{K}(c)|}$ also goes to zero as $c$ goes to infinity, which implies the desired result.
\end{proof}

Finally, we use the following inequality based on the Wallis formula for $\pi$ \cite{Wallis} to estimate the central binomial coefficient appearing in Lemma \ref{lem:totinequality} and in the proof of Theorem \ref{thm:avgsignature}. 

\begin{lemma}[Hirschhorn \cite{Hirschhorn}]
    \label{lem:Wallis}
    For any positive integer $m$,
    \[\frac{4^m}{\sqrt{\pi m}}\left(1-\frac{1}{4m}\right) < \binom{2m}{m} < \frac{4^m}{\sqrt{\pi m}}.\]
\end{lemma}

\section{4-genus}
\label{sec:4-genus}

The signature of a knot provides a lower bound for the 4-genus of the knot.  The rest of the paper focuses on an upper bound for the average 4-genus of a 2-bridge knot.

The proof technique is inspired by the proof of \cite[Lemma 2]{BKLMR} with a few notable changes.  Baader, Kjuchukova, Lewark, Misev, and Ray take $n$ to be half the number of crossings in a diagram of a 2-bridge knot and let $n$ tend to infinity.
Since the diagram need not be alternating, this diagrammatic crossing number need not be minimal. Because the variable $n$ satisfies $n+1\leq c\leq 2n$ where $c$ is the crossing number of the knot, their variable closely approximates the crossing number. In their model for 2-bridge knots, there is an even number of twist regions each with an even number of crossings, but both positive and negative twist regions are allowed. In their model, computing the genus is straightforward: the genus of such a knot is half the number of twist regions. However, working with the minimal crossing number is more difficult using their model.

Below we use a local modification of an oriented diagram called a \textit{saddle move}, that replaces two adjacent and oppositely-oriented arcs in a diagram as in Figure \ref{fig:saddle}. If $D_1$ and $D_2$ are diagrams of the knots $K_1$ and $K_2$, respectively, such that $D_1$ can be transformed into $D_2$ via $2k$ saddle moves and any number of Reidemeister moves (in any order), then there is a cobordism of genus $k$ between $K_1$ and $K_2$.
\begin{figure}[h!]
\begin{tikzpicture}[thick, decoration={
    markings,
    mark=at position 0.55 with {\arrow{>}}}]
    \draw[postaction={decorate}] (0,0) to [out = 30, in = 270] 
    (.3,.5) to [out = 90, in = 330]
    (0,1);
    \draw[postaction={decorate}] (1,1) to [out = 210, in = 90]
    (.7,.5) to [out = 270, in = 150]
    (1,0);

    \draw[<->] (2,.5) -- (3,.5);

    \draw[postaction={decorate}] (4,0) to [out = 60, in = 180] 
    (4.5,.3) to [out = 0, in = 120]
    (5,0);
     \draw[postaction={decorate}] (5,1) to [out = 240, in = 0]
    (4.5,.7) to [out = 180, in = 300]
    (4,1);
    
\end{tikzpicture}
\caption{An oriented saddle move.}
\label{fig:saddle}
\end{figure}
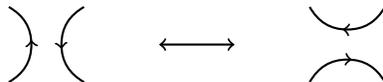

In Subsection \ref{subsec:newmodel}, we select a word $w$ at random from $T(2m+1)\cup T(2m+2)$, and let $D_w$ be its associated alternating diagram for the 2-bridge knot.  For $c=2m+j$ with $j=1$ or $2$, we let $c=st+r$ satisfying $0\leq r-1-j\leq s-1$.  Let $s<2m$ be a constant much smaller than $c$ so that we can partition the $c$ crossings into $t$ subsets of size $s$ with a remaining $r$ crossings.  In Proposition \ref{prop:newii} we perform oriented saddle moves as in Figure \ref{fig:cobordisms} in $t+1$ locations to ``cut'' the knot into $t+2$ connected summands via a cobordism of genus at most $t+1$.  The zeroth connected summand has only a single crossing that we remove by applying a Reidemeister I move. The last connected summand has $r-1$ crossings since we do not perform a saddle move after the ($2m+1$)st crossing.  The remaining $t$ connected summands each have $s$ crossings, with the saddle moves performed immediately following the ($ks+1$)st crossing for $0\leq k \leq \left\lfloor\frac{2m-1}{s}\right\rfloor=t$.  These oriented saddle moves are determined at each of the $t+1$ locations based on the orientations of the three arcs at this location.

By Proposition \ref{prop:bijection}, each of these $t$ summands is a random word of length $s$ in the alphabet $\{\sigma_1,\sigma_2^{-1}\}$.  There are $2^s$ such words possible.  These summands inherit orientations from the orientation of the original diagram $D_w$ associated with $w$ in $T(2m+1)\cup T(2m+2)$, which is prescribed.  There are $3\cdot 2^s$ such oriented summands possible.

 For each oriented connected summand, we identify in Definition \ref{def:mirror} its mirror image by rotating the diagram $180^\circ$, changing each crossing, and reversing the orientation.  Since every 2-bridge knot is equivalent to its reverse, we can reverse the orientation without changing the knot type. Let $2d=2d(s)$ be the number of these oriented connected summands whose diagram is not fixed by the operation rotating by $180^\circ$, changing each crossing, and reversing the orientation; the other $a=a(s)$ diagrams are amphichiral and do not change under this sequence of operations. 
Since $3\cdot 2^{s}= 2d+a$, where $a$ is at most $2^{\frac{s}{2}}$, we have that $d$ is bounded from above by $3\cdot 2^{s-1}$.  In Proposition \ref{prop:iconnectsummirror} we demonstrate why we collect these summands in pairs: if $K$ is a knot and $-\overline{K}$ is its reverse mirror image, then the $4$-genus of the connected sum $K\#(-\overline{K})$ is zero.



Some of the $2d$ summands are links;  in Lemma \ref{lem:link} we turn those appearing among the $t$ summands into knots using at most $3t$ saddle moves. If the last connected summand has two components, it can be transformed into a knot with a single saddle move. Thus there is a cobordism of genus at most $\frac{3}{2}t + \frac{1}{2}$ transforming the link into a knot. These saddle moves do not alter the $2^s$ random crossings but instead add additional crossings.

In Subsection \ref{subsec:random}, we then consider on the lattice $\mathbb{Z}^d\times\mathbb{Z}^a_2$ a random walk with $t$ steps coming from the $t$ summands.  
Theorem \ref{thm:taxicab} gives an upper bound on the expected value $3 \sqrt{2^s t}+a$ of the distance (along the lattice) of a random walk after $t$ steps.  Each of these steps in actuality has at most $s+3$ crossings, giving a contribution to the upper bound of the expected 4-genus of $\frac{3}{2}(s+3)\sqrt{2^st}+\frac{1}{2}(s+3)a$ coming from the 3-genus upper bound and the number of crossings in Inequality \ref{eq:crossing}.  We also have the last summand with $r-1$ crossings left over, contributing $\frac{1}{2}(r-1)$ to the upper bound of the expected 4-genus.

Finally in Section \ref{subsec:analysis} we analyze this upper bound on the expected 4-genus
\[(t+1) +\left(\frac{3}{2}t  + \frac{1}{2}\right)+\frac{3}{2} (s+3)\sqrt{2^st}+\frac{1}{2} (s+3)2^{\frac{s}{2}}+\frac{1}{2}(r-1),\]
which we then show is bounded from above by the term $\frac{4.94c}{\log c}$, to obtain Theorem \ref{thm:avg4genus}, which gives an explicit upper bound of $\frac{9.75 c}{\log c}$ on the average $4$-genus of $2$-bridge knots with crossing number $c$.

\begin{example}
\label{ex:cobordism}
    Let $w=+--+-+-+--++-++-\in T(12)$, with its associated knot diagram $D_w$ shown in Figure \ref{fig:CobEx}, and let $s=3$.  Performing the saddle moves prescribed in Proposition \ref{prop:newii} in the blue highlighted regions after crossings 1, 4, 7, and 10 decomposes the diagram into the connected sum of four 2-bridge knots and links. The first crossing can be removed via a Reidemeister I move. The second and fourth summands are links, and so we perform the saddle moves prescribed by the Link Lemma \ref{lem:link} to turn them into knots. These moves are highlighted in green. The first and third summands (highlighted in purple) are mirror images of one another, and so their connected sum is slice. The above process is depicted in Figure \ref{fig:CobEx}.

    Following our algorithm for this specific example, while making smart decisions, the 4-genus of the knot $K_w=12a_{1023}$ with diagram $D_w$ is bounded above by 6.  We obtain this by $\left(\frac{6}{2}\right)+\left(\frac{1}{2}+\frac{1}{2}\right)+\left\lfloor\frac{4+0+1}{2}\right\rfloor$, where the first contribution comes from Proposition \ref{prop:newii}, the second comes from Lemma \ref{lem:link} and also from the final link component, and the last comes from the upper bound on 3-genus which is an upper bound on 4-genus as in Equation \eqref{eq:crossing}.  The actual 4-genus of $K_w=12a_{1023}$ is one.  The upper bound as obtained in the equation above would be much higher.  
\end{example}

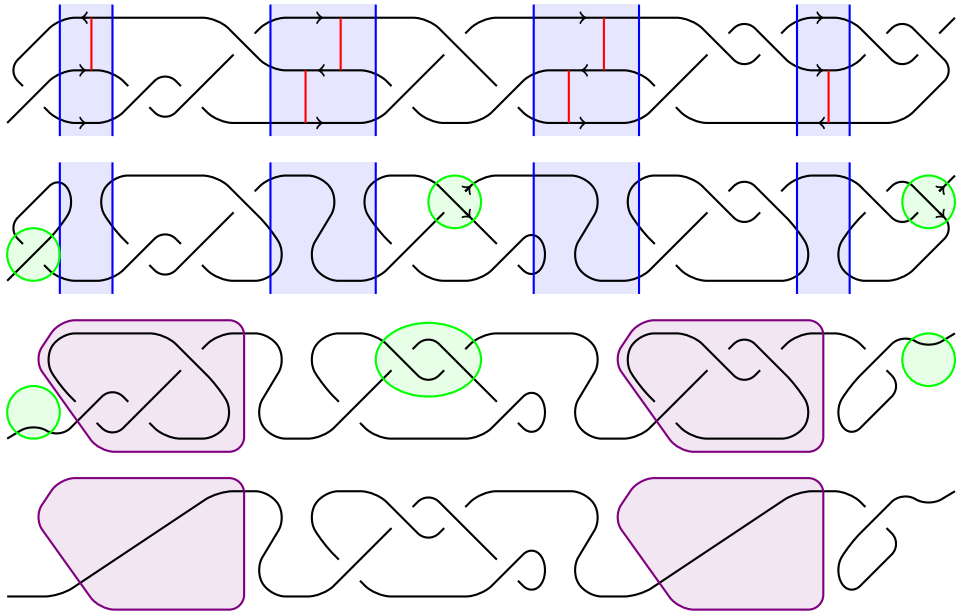
\begin{figure}[h!]
\[\begin{tikzpicture}[scale=.7,thick]

\fill[white!90!blue] (1,-.25) rectangle (2,2.25);
\fill[white!90!blue] (5,-.25) rectangle (7,2.25);
\fill[white!90!blue] (10,-.25) rectangle (12,2.25);
\fill[white!90!blue] (15,-.25) rectangle (16,2.25);

\begin{scope}[rounded corners = 2mm]
    \draw (0,0) -- (1,1) -- (2,1) -- (2.3,.7);
    \draw (2.7,.3) -- (3,0) -- (4.3,1.3);
    \draw (4.7,1.7) -- (5,2) -- (8,2) -- (9.3,.7);
    \draw (9.7,.3) -- (10,0) -- (12,0) -- (13.3,1.3);
    \draw (13.7,1.7) -- (14,2) -- (15,1) -- (16,1) -- (16.3,1.3);
    \draw (16.7,1.7) -- (17,2) -- (18,1) -- (17,0) -- (13,0) -- (12.7,.3);
    \draw (12.3,.7) -- (12,1) -- (10,1) -- (9,0) -- (8,0) -- (7.7,.3);
    \draw (7.3,.7) -- (7,1) -- (5,1) -- (4,2) -- (1,2) -- (0,1) -- (0.3,.7);
    \draw (.7,.3) -- (1,0) -- (2,0) -- (3,1) -- (3.3,.7);
    \draw (3.7,.3) -- (4,0) -- (7,0) -- (8.3,1.3);
    \draw (8.7,1.7) -- (9,2) -- (13,2) -- (14,1) -- (14.3,1.3);
    \draw (14.7,1.7) -- (15,2) -- (16,2) -- (17,1) -- (17.3,1.3);
    \draw (17.7,1.7) -- (18,2);
\end{scope}

\draw[blue] (1,2.25) -- (1,-.25);
\draw[blue] (2,2.25) -- (2,-.25);
\draw[blue] (5,2.25) -- (5,-.25);
\draw[blue] (7,2.25) -- (7,-.25);
\draw[blue] (10,2.25) -- (10,-.25);
\draw[blue] (12,2.25) -- (12,-.25);
\draw[blue] (15,2.25) -- (15,-.25);
\draw[blue] (16,2.25) -- (16,-.25);

\draw[<-] (1.4,2) -- (1.5,2);
\draw[->] (1.4,1) -- (1.5,1);
\draw[->] (1.4,0) -- (1.5,0);

\draw[->] (5.9,2) -- (6,2);
\draw[<-] (5.9,1) -- (6,1);
\draw[->] (5.9,0) -- (6,0);

\draw[->] (10.9,2) -- (11,2);
\draw[<-] (10.9,1) -- (11,1);
\draw[->] (10.9,0) -- (11,0);

\draw[->] (15.4,2) -- (15.5,2);
\draw[->] (15.4,1) -- (15.5,1);
\draw[<-] (15.4,0) -- (15.5,0);

\draw[red] (1.6,1) -- (1.6,2);
\draw[red] (5+2/3,0) -- (5+2/3,1);
\draw[red] (5+4/3,1) -- (5+4/3,2);
\draw[red] (10+2/3,0) -- (10+2/3,1);
\draw[red] (10+4/3,1) -- (10+4/3,2);
\draw[red] (15.6,0) -- (15.6,1);

\begin{scope}[yshift=-3cm]
\fill[white!90!blue] (1,-.25) rectangle (2,2.25);
\fill[white!90!blue] (5,-.25) rectangle (7,2.25);
\fill[white!90!blue] (10,-.25) rectangle (12,2.25);
\fill[white!90!blue] (15,-.25) rectangle (16,2.25);
    \fill[white!90!green] (8.5,1.5) circle (.5cm);
   \fill[white!90!green] (0.5,0.5) circle (.5cm);
   \fill[white!90!green](17.5,1.5) circle (.5cm);

\begin{scope}[rounded corners = 2mm]
    \draw (0,0) -- (1,1) -- (1.3,1.5) -- (1,2) -- (0,1) -- (.3,.7);
    \draw (.7,.3) -- (1,0) -- (2,0) -- (3,1) -- (3.3,.7);
    \draw (3.7,.3) -- (4,0) -- (5,0) -- (5.3,.5) -- (5,1) -- (4,2) -- (2,2) -- (1.7,1.5) -- (2,1) -- (2.3,.7);
    \draw (2.7,.3) -- (3,0) -- (4.3,1.3);
    \draw (4.7,1.7) -- (5,2) -- (6,2) -- (6.3,1.5) -- (6,1) -- (5.7,.5) -- (6,0) -- (7,0) -- (8.3,1.3);
    \draw (8.7,1.7) -- (9,2) -- (11,2) -- (11.3,1.5) -- (11,1) -- (10.7,.5) -- (11,0) -- (12,0) -- (13.3,1.3);
    \draw (13.7,1.7) -- (14,2) -- (15,1) -- (15.3,.5) -- (15,0) --(13,0) -- (12.7,.3);
    \draw (12.3,.7) -- (12,1) -- (11.7,1.5) -- (12,2) -- (13,2) -- (14,1) -- (14.3,1.3);
    \draw (14.7,1.7) -- (15,2) -- (16,2) -- (17,1) -- (17.3,1.3);
    \draw (17.7,1.7) -- (18,2);
    \draw (9.7,.3) -- (10,0) -- (10.3,.5) -- (10,1) -- (9,0) -- (8,0) -- (7.7,.3);
    \draw (7.3,.7) -- (7,1) -- (6.7,1.5) -- (7,2) -- (8,2) -- (9.3,.7);
    \draw (16.7,1.7) -- (17,2)-- (18,1) -- (17,0) -- (16,0) -- (15.7,.5) -- (16,1) -- (16.3,1.3);
\end{scope}

    \draw[blue] (1,2.25) -- (1,-.25);
    \draw[blue] (2,2.25) -- (2,-.25);
    \draw[blue] (5,2.25) -- (5,-.25);
    \draw[blue] (7,2.25) -- (7,-.25);
    \draw[blue] (10,2.25) -- (10,-.25);
    \draw[blue] (12,2.25) -- (12,-.25);
    \draw[blue] (15,2.25) -- (15,-.25);
    \draw[blue] (16,2.25) -- (16,-.25);

    \draw[->] (8.7,1.7) -- (8.8,1.8);
    \draw[->] (8.7,1.3) -- (8.8,1.2);

    \draw[->] (17.7,1.7) -- (17.8,1.8);
    \draw[->] (17.7, 1.3) -- (17.8,1.2);

    \draw[green] (8.5,1.5) circle (.5cm);
    \draw[green] (0.5,0.5) circle (.5cm);
    \draw[green] (17.5,1.5) circle (.5cm);

\end{scope}


\begin{scope}[yshift=-6cm]

    \fill[white!90!green] (.5,.5) circle (.5cm);
    \fill[white!90!green] (17.5,1.5) circle (.5cm);
   \fill[white!90!green] (8,1.5) ellipse (1cm and .7cm);
       \fill[rounded corners = 2mm, white!90!purple] (13,2.25) -- (12,2.25) -- (11.5,1.5) -- (12.75,-.25) -- (15.5,-.25) -- (15.5,2.25) -- (13,2.25);

       \fill[rounded corners = 2mm, white!90!purple] (2,2.25) -- (1,2.25) -- (.5,1.5) -- (1.75,-.25) -- (4.5,-.25) -- (4.5,2.25) -- (2,2.25);

    \begin{scope}[rounded corners = 2mm]
        \draw (0,0) -- (0.5,0.3) -- (1,0) -- (2,1) -- (2.3,.7);
        \draw (2.7,.3) --(3,0) -- (4,0) -- (4.3,.5) -- (4,1) -- (3,2) -- (1,2) -- (.7,1.5) -- (1,1) -- (1.3,.7);
        \draw (1.7,.3) -- (2,0) -- (3.3,1.3);
        \draw (3.7,1.7) -- (4,2) -- (5,2) -- (5.3,1.5) -- (5,1) -- (4.7,.5) -- (5,0) -- (6,0) -- (7.3,1.3);
        \draw (7.7,1.7) -- (8,2) -- (9.3,.7);
        \draw (9.7,.3) -- (10,0) -- (10.3,.5) -- (10,1) -- (9,0) -- (7,0) -- (6.7,.3);
        \draw (6.3,.7) -- (6,1) -- (5.7,1.5) -- (6,2) -- (7,2) -- (8,1) -- (8.3,1.3);
        \draw (8.7,1.7) -- (9,2) -- (11,2) -- (11.3,1.5) -- (11,1) -- (10.7,.5) -- (11,0) -- (12,0) -- (13.3,1.3);
        \draw (13.7,1.7) -- (14,2) -- (15,1) -- (15.3,.5) -- (15,0) -- (13,0) -- (12.7,.3);
        \draw (12.3,.7) -- (12,1) -- (11.7,1.5) -- (12,2) -- (13,2) -- (14,1) -- (14.3,1.3);
        \draw (14.7,1.7) -- (15,2) -- (16,2) -- (16.3,1.7);
        \draw (16.7,1.3) -- (17,1) -- (16,0) -- (15.7,.5) -- (16,1) -- (17,2) -- (17.5,1.7) -- (18,2);
    \end{scope}

    \draw[green] (.5,.5) circle (.5cm);
    \draw[green] (17.5,1.5) circle (.5cm);
    \draw[green] (8,1.5) ellipse (1cm and .7cm);

    \draw[rounded corners = 2mm, purple] (2,2.25) -- (1,2.25) -- (.5,1.5) -- (1.75,-.25) -- (4.5,-.25) -- (4.5,2.25) -- (2,2.25);
    \draw[rounded corners = 2mm, purple] (13,2.25) -- (12,2.25) -- (11.5,1.5) -- (12.75,-.25) -- (15.5,-.25) -- (15.5,2.25) -- (13,2.25);

\end{scope}


\begin{scope}[yshift = -9cm, rounded corners = 2mm]

       \fill[rounded corners = 2mm, white!90!purple] (13,2.25) -- (12,2.25) -- (11.5,1.5) -- (12.75,-.25) -- (15.5,-.25) -- (15.5,2.25) -- (13,2.25);

       \fill[rounded corners = 2mm, white!90!purple] (2,2.25) -- (1,2.25) -- (.5,1.5) -- (1.75,-.25) -- (4.5,-.25) -- (4.5,2.25) -- (2,2.25);

\draw (0,0) -- (1,0) -- (4,2) -- (5,2) -- (5.3,1.5) -- (5,1) -- (4.7,.5) -- (5,0) -- (6,0) -- (7.3,1.3);
 \draw (7.7,1.7) -- (8,2) -- (9.3,.7);
\draw (9.7,.3) -- (10,0) -- (10.3,.5) -- (10,1) -- (9,0) -- (7,0) -- (6.7,.3);
\draw (6.3,.7) -- (6,1) -- (5.7,1.5) -- (6,2) -- (7,2) -- (8,1) -- (8.3,1.3);
\draw (8.7,1.7) -- (9,2) -- (11,2) -- (11.3,1.5) -- (11,1) -- (10.7,.5) -- (11,0) -- (12,0) --  (15,2) -- (16,2) -- (16.3,1.7);
    \draw (16.7,1.3) -- (17,1) -- (16,0) -- (15.7,.5) -- (16,1) -- (17,2) -- (17.5,1.7) -- (18,2);

        \draw[rounded corners = 2mm, purple] (2,2.25) -- (1,2.25) -- (.5,1.5) -- (1.75,-.25) -- (4.5,-.25) -- (4.5,2.25) -- (2,2.25);
    \draw[rounded corners = 2mm, purple] (13,2.25) -- (12,2.25) -- (11.5,1.5) -- (12.75,-.25) -- (15.5,-.25) -- (15.5,2.25) -- (13,2.25);

\end{scope}

\end{tikzpicture}\]
\caption{The word $w=+--+-+-+--++-++- \in T(12)$ corresponding to the knot $12a_{1023}$ and the diagrams obtained from it via cobordisms as in our algorithm.}
\label{fig:CobEx}
\end{figure}

\subsection{Decomposing a knot into connected summands}
\label{subsec:newmodel}  
Instead of selecting a word $w$ at random from $T(2m+1)\cup T(2m+2)$, we establish a bijection that allows us to choose a word $w'$ at random from the set $ \{\sigma_1,\sigma_2^{-1}\}^{2m-1}$, removing the global condition of length modulo 3 and showing that indeed we have independent and identically distributed random variables for the individual crossings, which will become necessary in Subsection \ref{subsec:random}.

\begin{proposition}
\label{prop:bijection}    
There is a bijection between $T(2m+1)\cup T(2m+2)$ and $\{\sigma_1,\sigma_2^{-1}\}^{2m-1}$, the set of words of length $2m-1$ in $\{\sigma_1,\sigma_2^{-1}\}$.
\end{proposition}

\begin{proof}
Due to Fact \ref{fact:Jn}, with $|T(c)|=|J(c-2)|$ and $|\{\sigma_1,\sigma_2^{-1}\}^{2m-1}|=2^{2m-1}$, we need only show an injection.

Let $w\in T(2m+1)\cup T(2m+2)$. If $w\in T(2m+1)$, then $w=+w_1+$ for some word $w_1$. If $w\in T(2m+2)$, then $w=+w_2+-$ or $w=+w_3++-$ for some words $w_2$ or $w_3$. Define $f:T(2m+1)\cup T(2m+2)\to \{\sigma_1,\sigma_2^{-1}\}^{2m-1}$ by setting $f(w)$ to be the word obtained by first deleting the prefix and suffix from $w$ and then by replacing in $w_i$ runs $+$ and $--$ with $\sigma_1$ and runs $-$ and $++$ by $\sigma_2^{-1}$, as in Section \ref{sec:background}.

Since $w\in T(2m+1)\cup T(2m+2)$, its length is equivalent to one modulo three, and since $w$ is either $+w_1+$, $+w_2+-$, or $+w_3++-$, the lengths of $w_1$, $w_2$, and $w_3$ modulo three are 2, 1, and 0, respectively. Therefore, $f(+w_1+)$, $f(+w_2+-)$, and $f(+w_3++-)$ are all different words in $\{\sigma_1,\sigma_2^{-1}\}^{2m-1}$. Injectivity within each type of word is obvious.  Hence $f$ is injective, and therefore also bijective.
\end{proof}

In \cite[Lemma 2 (ii)]{BKLMR}, the authors perform $2(t-1)$ saddle moves at the start and end of two consecutive twist regions in order to transform a knot into $t$ connected summands as in Figure \ref{fig:BKLMR}.

\begin{figure}[!h]
\begin{center}
\begin{tikzpicture}

\begin{scope}[rounded corners = 3mm, thick]
    \draw (0,0) -- (1,1) -- (1.3,.7);
    \draw (1.7,.3) -- (2,0) -- (8,0) -- (9,1) -- (9.3,.7);
    \draw (9.7,.3) -- (10,0) -- (11,1) -- (11.3,.5) -- (11,0) -- (10.7,.3);
    \draw (10.3,.7) -- (10,1) -- (9,0) -- (8.7,.3);
    \draw (8.3,.7) -- (7,2) -- (6.7,1.7);
    \draw (6.3,1.3) -- (6,1) -- (5,2) -- (4.7,1.7);
    \draw (4.3,1.3) -- (4,1) -- (3,2) -- (2.7,1.5) -- (3,1) -- (3.3,1.3);
    \draw (3.7,1.7) -- (4,2) -- (5,1) -- (5.3,1.3);
    \draw (5.7,1.7) -- (6,2) -- (7,1) -- (7.3,1.3);
    \draw (7.7,1.7) -- (8,2) -- (12,2) -- (13,1) -- (13.3,1.3);
    \draw (12.3,1.3) -- (12,1) -- (11.7,.5) -- (12,0) -- (14,0);
    \draw (12.7,1.7) -- (13,2) -- (14,1);
    \draw (13.7,1.7) -- (14,2);
    \draw (0,1) -- (.3,.7);
    \draw (.7,.3) -- (1,0) -- (2,1) -- (2.3,1.5) -- (2,2) -- (0,2);

\end{scope}

\end{tikzpicture}
\end{center}
\caption{\label{fig:BKLMR} Two saddle moves at the start and end of consecutive twist regions, taken from \cite[Figure 3]{BKLMR}, with crossings matching our setting.}
\end{figure}
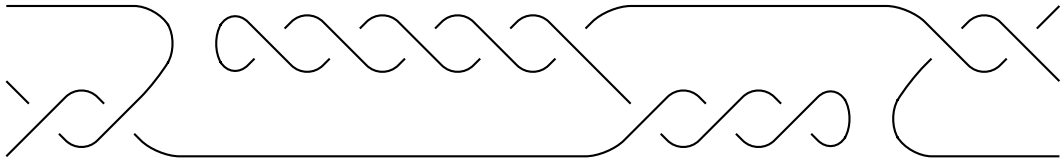

In our work instead, we will use cobordisms of genus at most $t+1$ to produce $t+1$ connected summands, with all but the last having $s$ crossings.  Each connected summand is determined by an oriented word, defined as follows.

\begin{definition}
\label{def:OW}
For every positive integer $s$, define $W(s)$ to be the set $\{\sigma_1,\sigma_2^{-1}\}^s$ of words of size $s$.  Define $OR=\{o_1, o_2,o_3\}$ to be the set of the three possible orientations, with $o_i$ having the $i$th strand oriented to the left.  Let $OW(s)=OR\times W(s)$ be the set of oriented words of size $s$. For brevity, we often denote the oriented word $(o_i,z)\in OW(s)$ by $\vec{\omega}$. Each oriented word $\vec{\omega}$ has an associated connected summand diagram (or long knot diagram) $D_{\vec{\omega}}$ obtained by capping off two adjacent and oppositely oriented strands on the left and right of the standard diagram of the oriented word $\vec{\omega}$, as determined in Figure \ref{fig:orwords}. The top row of Figure \ref{fig:orwords} shows how to cap off the strands to the left of $\vec{\omega}$, and the bottom row shows how to cap off the strands to the right of $\vec{\omega}$.
\end{definition}

\begin{figure}[h!]
\begin{tikzpicture}[thick]

\draw (2,-.5) -- (1,-.5) -- (1,2.5) -- (2,2.5);

\draw[->] (1,2) -- (.5,2);
\draw (.5,2) -- (0,2);

\draw[->] (0,1) -- (0.5,1);
\draw (0.5,1) -- (1,1);

\draw[->] (0,0) -- (0.5,0);
\draw (0.5,0) -- (1,0);

\draw[dashed] (0,-.5) -- (0,2.5);

    \begin{scope}[decoration={
    markings,
    mark=at position 0.5 with {\arrow{>}}}]]
        \draw[postaction={decorate}] (0,2) to [out=180, in = 180]  (0,1);
    \end{scope}
    
 \draw (-.3,0) -- (0,0);
 \node at (1.7,1) {$\vec{\omega}$};

 \begin{scope}[xshift = 4cm]
 \draw (2,-.5) -- (1,-.5) -- (1,2.5) -- (2,2.5);

\draw (1,2) -- (.5,2);
\draw[<-] (.5,2) -- (0,2);

\draw (0,1) -- (0.5,1);
\draw[<-] (0.5,1) -- (1,1);

\draw[->] (0,0) -- (0.5,0);
\draw (0.5,0) -- (1,0);

\draw[dashed] (0,-.5) -- (0,2.5);

    \begin{scope}[decoration={
    markings,
    mark=at position 0.5 with {\arrow{>}}}]]
        \draw[postaction={decorate}] (0,1) to [out=180, in = 180]  (0,2);
    \end{scope}
    
 \draw (-.3,0) -- (0,0);
 
 \node at (1.7,1) {$\vec{\omega}$};
 \end{scope}
 
 \begin{scope}[xshift = 8cm]
 \draw (2,-.5) -- (1,-.5) -- (1,2.5) -- (2,2.5);

\draw (1,2) -- (.5,2);
\draw[<-] (.5,2) -- (0,2);

\draw[->] (0,1) -- (0.5,1);
\draw (0.5,1) -- (1,1);

\draw (0,0) -- (0.5,0);
\draw[<-] (0.5,0) -- (1,0);

\draw[dashed] (0,-.5) -- (0,2.5);

    \begin{scope}[decoration={
    markings,
    mark=at position 0.5 with {\arrow{>}}}]]
        \draw[postaction={decorate}] (0,0) to [out=180, in = 180]  (0,1);
    \end{scope}
    
 \draw (-.3,2) -- (0,2);
 
 \node at (1.7,1) {$\vec{\omega}$};
 \end{scope}
 
 \begin{scope}[yshift =- 4cm, xshift = 1.7cm]

 \draw (-2,-.5) -- (-1,-.5) -- (-1, 2.5) -- (-2, 2.5);
 
 \draw[->] (0,2) -- (-.5,2);
 \draw (-.5,2) -- (-1,2);
 
 \draw[->] (-1,1) -- (-.5,1);
 \draw (-.5,1) -- (0,1);
 
 \draw[->] (-1,0) -- (-.5,0);
 \draw (-.5,0) -- (0,0);
 
     \begin{scope}[decoration={
    markings,
    mark=at position 0.5 with {\arrow{>}}}]]
        \draw[postaction={decorate}] (0,1) to [out=0, in = 0]  (0,2);
    \end{scope}
    
    \draw[dashed] (0,-.5) -- (0,2.5);

  \node at (-1.7,1) {$\vec{\omega}$};
  
   \draw (.3,0) -- (0,0);
 
 \end{scope}
  \begin{scope}[yshift =- 4cm, xshift = 5.7cm]

 \draw (-2,-.5) -- (-1,-.5) -- (-1, 2.5) -- (-2, 2.5);
 
 \draw(0,2) -- (-.5,2);
 \draw[<-] (-.5,2) -- (-1,2);
 
 \draw (-1,1) -- (-.5,1);
 \draw[<-] (-.5,1) -- (0,1);
 
 \draw[->] (-1,0) -- (-.5,0);
 \draw (-.5,0) -- (0,0);
 
     \begin{scope}[decoration={
    markings,
    mark=at position 0.5 with {\arrow{>}}}]]
        \draw[postaction={decorate}] (0,0) to [out=0, in = 0]  (0,1);
    \end{scope}
    
    \draw[dashed] (0,-.5) -- (0,2.5);

  \node at (-1.7,1) {$\vec{\omega}$};
  
   \draw (.3,2) -- (0,2);
 
 \end{scope}
  \begin{scope}[yshift =- 4cm, xshift = 9.7cm]

 \draw (-2,-.5) -- (-1,-.5) -- (-1, 2.5) -- (-2, 2.5);
 
 \draw(0,2) -- (-.5,2);
 \draw[<-] (-.5,2) -- (-1,2);
 
 \draw[->] (-1,1) -- (-.5,1);
 \draw(-.5,1) -- (0,1);
 
 \draw (-1,0) -- (-.5,0);
 \draw[<-] (-.5,0) -- (0,0);
 
     \begin{scope}[decoration={
    markings,
    mark=at position 0.5 with {\arrow{>}}}]]
        \draw[postaction={decorate}] (0,1) to [out=0, in = 0]  (0,0);
    \end{scope}
    
    \draw[dashed] (0,-.5) -- (0,2.5);

  \node at (-1.7,1) {$\vec{\omega}$};
  
   \draw (.3,2) -- (0,2);
    
 \end{scope}

\end{tikzpicture}
\caption{\label{fig:orwords} The connected summand diagram $D_{\vec{\omega}}$ associated with the oriented word $\vec{\omega}$.}
\end{figure}

\begin{proposition}
\label{prop:newii}
Let $w \in T(2m+1) \cup T(2m+2)$, let $D_w$ be its associated alternating diagram, and let $K_w$ be the knot represented by $D_w$. Then $K_w$ admits a cobordism of genus at most $t+1$, realized by at most $2t+2$ saddle moves on $D_w$, to a connected sum
\[
L_1 \# L_2 \# \cdots \# L_{t+1},
\]
where each summand $L_i$ is the unknot, a $2$-bridge knot, or a $2$-bridge link.
\end{proposition}

\begin{figure}[h!]
\begin{tikzpicture}[thick]
    \draw[->] (0,0) -- (3,0);
    \draw[->] (3,1) -- (0,1);
    \draw[->] (0,2) -- (3,2);
    \draw[red] (1,0) -- (1,1);
    \draw[red] (2,1) -- (2,2);

    \draw[->] (1.5,-.5) -- (1.5,-1.5);
    \begin{scope}[yshift = -4cm, decoration={
    markings,
    mark=at position 0.5 with {\arrow{>}}}]]
        \draw[postaction={decorate}] (0,0) to (.5,0)
        to [out = 0, in = 0] (.5,1) to (0,1);
        \draw[postaction={decorate}] (0,2) to (1.5,2) to [out = 0, in = 0] (1.5,1) to [out = 180, in = 180] (1.5,0) to (3,0);
        \draw[postaction={decorate}] (3,1) to (2.5,1) to [out=180, in = 180] (2.5,2) to (3,2);
    \end{scope}

    \begin{scope}[xshift = 4cm]
        \draw[->](0,2) -- (2,2);
        \draw[->](0,1) -- (2,1);
        \draw[->](2,0) -- (0,0);
        \draw[red] (1,1) -- (1,0);

        \draw[->] (1,-.5) -- (1,-1.5);
    \end{scope}
    \begin{scope}[xshift = 4cm, yshift = -4cm,decoration={
    markings,
    mark=at position 0.525 with {\arrow{>}}}]]
        \draw[postaction={decorate}] (0,1) to (.5,1) to [out = 0, in = 0] (.5,0) to (0,0);
        \draw[postaction={decorate}] (2,0) to (1.5,0) to [out = 180, in = 180] (1.5,1) to (2,1);
        \draw[postaction={decorate}] (0,2) to (2,2);
    \end{scope}

    \begin{scope}[xshift = 7cm]
        \draw[->](2,2) -- (0,2);
        \draw[->](0,1) -- (2,1);
        \draw[->](0,0) -- (2,0);
        \draw[red] (1,1) -- (1,2);

        \draw[->] (1,-.5) -- (1,-1.5);
    \end{scope}
    \begin{scope}[xshift = 7cm, yshift = -4cm,decoration={
    markings,
    mark=at position 0.51 with {\arrow{<}}}]
        \draw[postaction={decorate}] (0,2) to (0.5,2) to [out = 0, in = 0] (.5,1) to (0,1);
        \draw[postaction={decorate}] (2,1) to (1.5,1) to [out = 180, in = 180] (1.5,2) to (2,2);
        \draw[postaction={decorate}] (2,0) to (0,0);
    \end{scope}
    
\end{tikzpicture}
\caption{\label{fig:cobordisms}
Saddle moves determined by orientations of the strands after each $s$ crossings.  }
\end{figure}
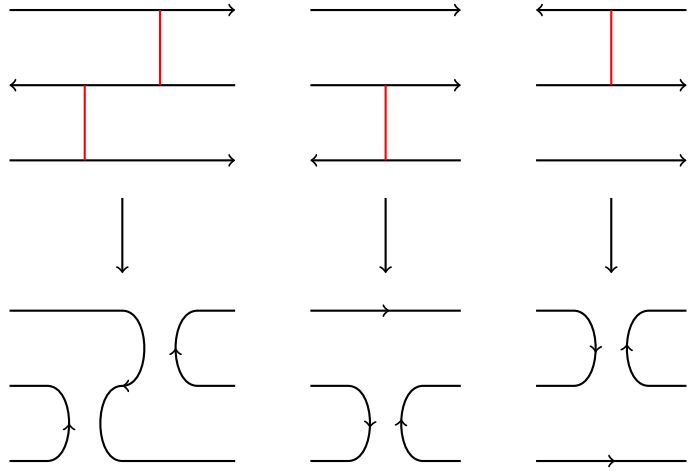

\begin{proof}
The word $w$ gives an oriented alternating diagram $D_w$ of $K_w$. For $0\leq k \leq \left\lfloor\frac{2m-1}{s}\right\rfloor=t$, define $\vec{\omega}_k\in OW(s)$ to be the oriented word consisting of the portion of $D_w$ after crossing $1+(k-1)s$ and before crossing $1+ks$. Define $\vec{\omega}_{t+1}$ to be the oriented braid word consisting of the portion of the diagram after crossing $1+ts$ until immediately after the final crossing.  

Performing the saddle moves depicted in Figure \ref{fig:cobordisms} following crossing $1+ks$ for $0\leq k \leq t$ transforms $D_w$ into the diagram $U\#D_{\vec{\omega}_1}\# D_{\vec{\omega}_2}\#\cdots\# D_{\vec{\omega}_t}\#D_{\vec{\omega}_{t+1}}$ where the $D_{\vec{\omega}_i}$ are defined in Definition \ref{def:OW} and $U$ is a single crossing that can be removed via a Reidemeister I move. Since each cobordism described in Figure \ref{fig:cobordisms} consists of at most two saddle moves, the total number of saddle moves is at most $2t+2$. Thus there is a cobordism of genus at most $t+1$ between $K_w$ and $L_1\# L_2\# \cdots \# L_t \# L_{t+1}$ where  $L_k$ is the knot or link with diagram $D_{\vec{\omega}_k}$ for $1\leq k \leq t+1$.

\end{proof}

\begin{remark}
\label{rem:cutting}
There is a cobordism of genus at most $t+1$ from the knot $K_w$ to the connected sum of links $L_1\#L_2\# \cdots \# L_t \# L_{t+1}$ contributing $t+1$ to the upper bound for expected 4-genus. We will keep track of these contributions to the upper bound for the expected 4-genus in remarks after results as we proceed.
\end{remark}

In \cite[Lemma 2 (i)]{BKLMR}, the authors pair up connected summands that are mirror images of each other; we do the same below.  The reason for this is the following well-known result on 4-genus.

\begin{proposition}
\label{prop:iconnectsummirror}
For all knots $K$, the connected sum of $K$ with its reverse mirror image $-\overline{K}$ is a ribbon knot; in particular, $g_4(K\#(-\overline{K}))=0$.
\end{proposition}

Proposition \ref{prop:iconnectsummirror} motivates us to keep track of which connected summands occur together with their reverse mirror images. In order to do this combinatorially, we need to describe the mirror operation directly on the oriented words that encode the summands obtained in Proposition \ref{prop:newii}. The operation must account not only for changing each crossing, but also for the fact that the summand diagrams inherit orientations from the original diagram. The following definition gives this operation, allowing us later to group summands into mirror pairs whose contributions to the 4-genus can be discarded.

\begin{definition}
\label{def:mirror}
Let $\vec{\omega}\in OW(s)$ be an oriented word, and let $D_{\vec{\omega}}$ be its associated diagram obtained as in Figure \ref{fig:orwords}. Define $\mirror(D_{\vec{\omega}})$ to be the diagram obtained by rotating $D_{\vec{\omega}}$ by $180^\circ$, changing every crossing, and reversing the orientation. Define $\mirror(\vec{\omega})$ to be the oriented word in $OW(s)$ whose diagram is $D_{\mirror(\vec{\omega})}=\mirror(D_{\vec{\omega}})$. The links with diagrams $D_{\vec{\omega}}$ and $\mirror(D_{\vec{\omega}})$ are mirror images of one another.

If $\vec{\omega} = (o_i,z)$ where $z \in W(s)$ is an unoriented word, then we define $\mirror(z) \in W(s)$ to be the unoriented word such that $\mirror(\vec{\omega}) = (o_j,\mirror(z))$ for some $j=1$, $2$, or $3$. The word $\mirror(z)$ is obtained from $z$ by reversing the word and exchanging every $\sigma_1$ with a $\sigma_2^{-1}$ and vice versa.
\end{definition}

In Example \ref{ex:cobordism}, there are three connected summands corresponding to the oriented words $\vec{\omega}_1=(o_1,\sigma_1^2\sigma_2^{-1})$, $\vec{\omega}_2=(o_2,\sigma_1\sigma_2^{-1}\sigma_1)$, and $\vec{\omega}_3=(o_2,\sigma_1\sigma_2^{-2})$.  The first and third are mirror images of each other since the orientation to the right of $\vec{\omega}_1$ is $o_2$ and reversing $\sigma_1^2\sigma_2^{-1}$ and interchanging $\sigma_1$ and $\sigma_2^{-1}$  results in $\sigma_1\sigma_2^{-2}$, it follows that $\vec{\omega}_3=\mirror(\vec{\omega}_1)$.

It may be the case that some of our connected summands are 2-component links rather than knots. We transform each 2-component link summand into a knot summand via an operation $\merge$ described in Definition \ref{def:merge} below.  Figure \ref{fig:aux} shows two auxiliary moves we use to define $\merge$.  The first, which we call a twist, takes two coherently oriented strands and adds a crossing between them using a Reidemeister I move and a saddle move.  The second changes any oriented crossing using two saddle moves and two Reidemeister I moves.

\begin{figure}[h!]
\begin{tikzpicture}[thick]

\draw[->] (1.5,.5) -- (2,.5);
\draw[->] (5,.5) -- (5.5,.5);
\draw[->] (8,.5) -- (8.5,.5);

\begin{scope}[yshift = 2cm]

\draw (-2,.5) node{Twist};

\draw[->] (1.5,.5) -- (2,.5);
\draw[->] (5,.5) -- (5.5,.5);

\draw[->, rounded corners=3mm] (0,0) -- (0.3,0.5) -- (0,1);
\draw[->, rounded corners = 3mm] (1,0) -- (0.7,0.5) -- (1,1);

\begin{scope}[xshift = 2.5cm]
    \begin{knot}[
consider self intersections,
clip width = 4,
end tolerance = 2pt
]
\strand[->] (0,0) to [out = 45, in = 90,looseness=2] (1.25,.5) to [out = -90, in = -45, looseness = 2] (0,1);
\end{knot}
\draw [->] (1.75,0) to [out = 120, in =240] (1.75,1);
\draw[red] (1.25,.5) -- (1.6,.5);
\end{scope}

\begin{scope}[xshift = 6cm]
\begin{knot}[
consider self intersections,
clip width = 4,
end tolerance = 2pt
]
\strand[->] (0,0) -- (1,1);
\strand[->] (1,0) -- (0,1);
\end{knot}
\end{scope}

\end{scope}


\draw (-2,.5) node{Crossing change};

\begin{knot}[
consider self intersections,
clip width = 4,
end tolerance = 2pt
]
\strand[->] (0,0) -- (1,1);
\strand[->] (1,0) -- (0,1);
\end{knot}
\draw[red] (.9,.9) -- (.9,.1);

\begin{scope}[xshift = 2.5cm]
    \begin{knot}[
consider self intersections,
clip width = 4,
end tolerance = 2pt
]
\strand[->] (0,0) to [out = 45, in = 90,looseness=2] (1.25,.5) to [out = -90, in = -45, looseness = 2] (0,1);
\end{knot}
\draw [->] (1.75,0) to [out = 120, in =240] (1.75,1);
\end{scope}

\begin{scope}[xshift = 6cm]
    \begin{knot}[
consider self intersections,
clip width = 4,
end tolerance = 2pt
]
\flipcrossings{1};
\strand[->] (0,0) to [out = 45, in = 90,looseness=2] (1.25,.5) to [out = -90, in = -45, looseness = 2] (0,1);
\end{knot}
\draw [->] (1.75,0) to [out = 120, in =240] (1.75,1);
\draw[red] (1.25,.5) -- (1.6,.5);
\end{scope}

\begin{scope}[xshift = 9cm]
\begin{knot}[
consider self intersections,
clip width = 4,
end tolerance = 2pt
]
\strand[->] (1,0) -- (0,1);
\strand[->] (0,0) -- (1,1);
\end{knot}
\end{scope}

\end{tikzpicture}

\caption{Two auxiliary moves. The top row shows how to introduce a crossing between two coherently oriented strands using a Reidemeister I move and a saddle move. The second row shows how to change a crossing using two saddle moves and two Reidemeister I moves.}
\label{fig:aux}
\end{figure}

\begin{definition}
\label{def:merge}
Let $\vec{\omega}=(o_i,z)\in OW(s)$. Suppose its associated diagram $D_{\vec{\omega}}$ has two components. If $s$ is even, the orientations of the three strands between crossings $\frac{s}{2}$ and $\frac{s}{2}+1$ of $\vec{\omega}$ are determined by whether the top, middle, or bottom strand faces to the left.  If the top or bottom strand faces to the left, then two adjacent strands are coherently oriented to the right, and the operation $\merge$ performs a twist move as in the first row of Figure \ref{fig:linklemma}. Otherwise, the middle strand faces to the left, and $\merge$ performs a Reidemeister II move on two incoherently oriented strands, followed by a twist move, and then finally a crossing change as in the second and third rows of Figure \ref{fig:linklemma}. In this case, if $z$ is less than $\mirror(z)$ in lexicographical order, then  $\merge$ performs the moves in the second row, and if $\mirror(z)$ is less than $z$ in lexicographical order, then $\merge$ performs the moves in the third row. We remark that in this case it is impossible for $\mirror(z)=z$ for the following reason. If $\mirror(z)=z$, then $D_{\vec{\omega}}$ and $\mirror(D_{\vec{\omega}})$ are the same diagram, and thus their signature is zero. However, every 2-bridge link with two components has odd signature, a contradiction. 

When $s$ is odd, consider the orientations of the three strands at the central crossing.  If the crossing has two strands oriented to the right, $\merge$ performs the twist move to create a second copy of the crossing as in the fourth row of Figure \ref{fig:linklemma}.  Otherwise, on exactly one side of the central crossing, there are two coherently oriented strands:  they appear on the right side in the fifth row of Figure \ref{fig:linklemma} and they appear on the left side in the sixth row of the figure.  In both cases, $\merge$ performs the twist move to create the other crossing.

\end{definition}

%
%

\begin{lemma}
\label{lem:link} \emph{Link Lemma.  } Let $\vec{\omega}\in OW(s)$ be an oriented word with associated diagram $D_{\vec{\omega}}$. Suppose that $D_{\vec{\omega}}$ has two components. Then $\merge(D_{\vec{\omega}})$ has only one component. Also, the operations $\mirror$ and $\merge$ commute.

\end{lemma}


\begin{figure}[h!]
\begin{tikzpicture}[thick]

\draw[purple, ultra thick] (.5,0) -- (.5,.5);
\draw[->] (0,0) -- (1,0);
\draw[->] (0,.5) -- (1,.5);
\draw[->] (1,1) -- (0,1);

\draw [->] (2,.5) -- (3,.5);

\begin{scope}[xshift = 4cm]

\draw[rounded corners=2mm, ->] (0,0) -- (.25,0) -- (.75,.5) -- (1,.5);
\draw[rounded corners = 2mm] (0,.5) -- (.25,.5) -- (.4,.35);
\draw[rounded corners = 2mm, ->] (.6,.15) -- (.75,0) -- (1,0);
\draw[->] (1,1) -- (0,1);
\end{scope}
\begin{scope}[yshift = -2cm]
\draw[->] (0,0) -- (1,0);
\draw[<-] (0,.5) -- (1,.5);
\draw[->] (0,1) -- (1,1);

\draw [->] (2,.5) -- (3,.5);
\draw (2.4,.7) node{\small{R II}};

\begin{scope}[xshift = 4cm]
\fill[white!80!purple] (0.75,0.25) circle (.2cm);
\draw[->, rounded corners = 2mm] (0,0) -- (.5,.5) -- (1,0);
\draw (1,.5) -- (.85,.35);
\draw[rounded corners = 2mm] (.65,.15) -- (.5,0) -- (.35,.15);
\draw[->] (.15,.35) -- (0,.5);
\draw[->] (0,1) -- (1,1);
\draw[purple, ultra thick] (0.5,.4) -- (0.5,1);

\draw [->] (2,.5) -- (3,.5);

\end{scope}

\begin{scope}[xshift = 8cm, rounded corners = 2mm]

\draw (0,0) -- (.65,.65);
\draw[->] (.85,.85) -- (1,1) -- (1.5,1);
\draw[<-] (0,.5) -- (.15,.35);
\draw (.35,.15) -- (.5,0) -- (1,0) -- (1.5,.5);
\draw (0,1) -- (.5,1) -- (1.15,.35);
\draw[->] (1.35,.15) -- (1.5,0);
    
\end{scope}

\end{scope}
\begin{scope}[yshift = -4cm]
\draw[->] (0,0) -- (1,0);
\draw[<-] (0,.5) -- (1,.5);
\draw[->] (0,1) -- (1,1);

\draw [->] (2,.5) -- (3,.5);
\draw (2.4,.7) node{\small{R II}};

\begin{scope}[xshift = 4cm]
\fill[white!80!purple] (0.75,0.75) circle (.2cm);
\draw[->, rounded corners = 2mm] (0,1) -- (.5,.5) -- (1,1);
\draw (1,.5) -- (.85,.65);
\draw[rounded corners = 2mm] (.65,.85) -- (.5,1) -- (.35,.85);
\draw[->] (.15,.65) -- (0,.5);
\draw[->] (0,0) -- (1,0);
\draw[purple, ultra thick] (0.5,.6) -- (0.5,0);

\draw [->] (2,.5) -- (3,.5);

\end{scope}

\begin{scope}[xshift = 8cm, rounded corners = 2mm]

\draw (0,1) -- (.65,.35);
\draw[->] (.85,.15) -- (1,0) -- (1.5,0);
\draw[<-] (0,.5) -- (.15,.65);
\draw (.35,.85) -- (.5,1) -- (1,1) -- (1.5,.5);
\draw (0,0) -- (.5,0) -- (1.15,.65);
\draw[->] (1.35,.85) -- (1.5,1);
    
\end{scope}

\end{scope}

\begin{scope}[yshift=-6cm]

\draw[rounded corners = 2mm, ->] (0,0) -- (.25,0) -- (.75,.5) -- (1,.5);
\draw[rounded corners = 2mm] (0,.5) -- (.25,.5) -- (.35,.35);
\draw[rounded corners = 2mm, ->] (.65,.15) -- (.75,0) -- (1,0);
\draw[->] (1,1) -- (0,1);
\draw[purple, ultra thick] (.8,.47) -- (.8,.05);

\draw [->] (2,.5) -- (3,.5);

\begin{scope}[xshift = 4cm, rounded corners = 2mm]
    \draw (0,0) -- (.5,.5) -- (.65,.35);
    \draw[->] (.85,.15) -- (1,0);
    \draw (0,.5) -- (.15,.35);
    \draw[->] (.35,.15) -- (.5,0) -- (1,.5);
    \draw[->] (1,1) -- (0,1);
\end{scope}

\end{scope}

\begin{scope}[yshift=-8cm]

\draw[rounded corners = 2mm, ->] (0,0) -- (.25,0) -- (.75,.5) -- (1,.5);
\draw[rounded corners = 2mm, <-] (0,.5) -- (.25,.5) -- (.35,.35);
\draw[rounded corners = 2mm] (.65,.15) -- (.75,0) -- (1,0);
\draw[<-] (1,1) -- (0,1);
\draw[purple, ultra thick] (.8,.47) -- (.8,1);

\draw [->] (2,.5) -- (3,.5);

\begin{scope}[xshift = 4cm, rounded corners = 2mm]
    \draw (0,0) -- (.65,.65); 
    \draw[->] (.85,.85) -- (1,1);
    \draw[<-] (0,.5) -- (.15,.35);
    \draw (.35,.15) -- (.5,0) -- (1,0);
    \draw[->] (0,1) -- (.5,1) -- (1,.5);
\end{scope}

\end{scope}

\begin{scope}[yshift=-10cm]

\draw[rounded corners = 2mm, <-] (0,0) -- (.25,0) -- (.75,.5) -- (1,.5);
\draw[rounded corners = 2mm] (0,.5) -- (.25,.5) -- (.35,.35);
\draw[rounded corners = 2mm,->] (.65,.15) -- (.75,0) -- (1,0);
\draw[<-] (1,1) -- (0,1);
\draw[purple, ultra thick] (.2,.47) -- (.2,1);

\draw [->] (2,.5) -- (3,.5);

\begin{scope}[xshift = 4cm, rounded corners = 2mm]
    \draw (0,1) -- (.65,.35); 
    \draw[->] (.85,.15) -- (1,0);
    \draw (0,.5) -- (.15,.65);
    \draw[->] (.35,.85) -- (.5,1) -- (1,1);
    \draw[<-] (0,0) -- (.5,0) -- (1,.5);
\end{scope}

\end{scope}

\draw (-3,1.5) rectangle (10,-10.5);
\draw (-3,-4.5) -- (10,-4.5);
\draw (-.5,1.5) -- (-.5,-10.5);
\draw (-1.75 ,-1.5) node {Even};
\draw (-1.75, -7.5) node {Odd};

\end{tikzpicture}
\caption{\label{fig:linklemma}
 Saddle moves can be used to turn connected summands that are 2-component links into knots, as in the Link Lemma \ref{lem:link}. Twist moves are indicated with a purple segment and crossing changes by highlighting the crossing.}
\end{figure}

\begin{proof}

If two strands of $D_{\vec{\omega}}$ are coherently oriented, then they must belong to distinct components; inserting a crossing between them, which $\merge$ does in every case above, creates a knot.

A straightforward check shows that the operations $\mirror$ and $\merge$ commute, that is, performing the $\merge$ operation in Figure \ref{fig:linklemma} to the link diagram $D_{\vec{\omega}}$ followed by rotating by $180^\circ$, changing the crossings, and reversing the orientation results in the same diagram as rotating by $180^\circ$, mirroring the crossings, and reversing the orientation followed by performing the $\merge$ operation in Figure \ref{fig:linklemma}. Perhaps the most interesting case in this straightforward check is when $s$ is even and the middle strand of $\vec{\omega}$ between crossings $\frac{s}{2}$ and $\frac{s}{2}+1$ points to the left. In this case $\merge$ performs the operation in the second row of Figure \ref{fig:linklemma} to one of $D_{\vec{\omega}}$ or $\mirror(D_{\vec{\omega}})$ and the operation in the third row of Figure \ref{fig:linklemma} to the other, ensuring symmetry under the $\mirror$ operation.
\end{proof}

\begin{remark}
\label{rem:link}
If the ($t+1$)st summand $L_{t+1}$ is a 2-component link, then since the last crossing is oriented to the right, we can perform the inverse of the twist move to remove this crossing and turn the summand into a knot $K_{t+1}$. Thus there is a cobordism of genus at most $\frac{3}{2}t + \frac{1}{2}$ consisting of at most $3t+1$ saddle moves from the link $L_1\#\cdots \#L_{t+1}$ in Proposition \ref{prop:newii} with up to $2t+2$ components to the knot $K_1\#\cdots \#K_{t+1}$.



\end{remark}

Performing the Link Lemma does not change the number $3\cdot 2^s$ of possible connected summands.

If the oriented word $\vec{\omega}=(o_i,z)$ satisfies $\vec{\omega}=\mirror(\vec{\omega})$, then the summand $K_{\vec{\omega}}$ with diagram $D_{\vec{\omega}}$ is amphichiral.   In order for this to occur, the number of crossings $s$ must be even, the unoriented word $z$ has to be invariant when reversing its order and swapping each $\sigma_1$ with a $\sigma_2^{-1}$ (and vice versa), and the middle strand between crossings $\frac{s}{2}$ and $\frac{s}{2}+1$ must be pointed to the left. We say such oriented words are of \textit{amphichiral type}. An oriented word not of amphichiral type is said to be of \textit{chiral type}. The number $a=a(s)$ of oriented words of amphichiral type is 0 when $s$ is odd and $2^\frac{s}{2}$ when $s$ is even. The number $d=d(s)$ of oriented braid words of chiral type satisfies $3\cdot 2^s=2d+a$, so that $d\leq 3\cdot 2^{s-1}$.


\subsection{Random walk on a lattice}  \label{subsec:random} Each pair of words of chiral type corresponding to knots that are mirror images of each other defines a dimension in the integer sublattice $\mathbb{Z}^d$ of $\mathbb{Z}^d\times\mathbb{Z}_2^a$.  For each pair of mirror image knots, we arbitrarily choose one.  When this word appears as a connected summand, we will take a step in the positive direction in that dimension.  When its mirror appears as a connected summand, we will take a step in the negative direction in that dimension.  Since a word of amphichiral type is equivalent to its own mirror, each of these will define a dimension in $\mathbb{Z}^a_2$.

The goal of this subsection is to show the following result:

\begin{theorem}
\label{thm:taxicab}
For all positive integers $s$ and $t$, the expected taxicab distance from the origin after a random walk as in Definition \ref{def:random} of $t$ steps on the lattice $\mathbb{Z}^d\times\mathbb{Z}_2^a$ with $2d+a=3\cdot 2^s$ is at most $3\sqrt{2^s t} +a$.
\end{theorem}

\begin{remark}
\label{rem:randomwalkdistance}
After applying the Link Lemma \ref{lem:link}, each connected summand has at most $s+3$ crossings, yielding an upper bound on 4-genus by $\frac{s+3}{2}$ by Inequality \ref{eq:crossing}.  The result above implies that the expected taxicab distance of $3\sqrt{2^s t} +a$ is the number of connected summands that are not canceled by their mirror.  Then this contributes $\frac{1}{2}(s+3)(3\sqrt{2^s t} +a)$ to the upper bound for the expected 4-genus of the knot $K_w$.

Additionally, there are at most $r-1$ crossings in the last connected summand, contributing an additional $\frac{1}{2}(r-1)$ to the upper bound for the expected 4-genus of the knot $K_w$ by Inequality \ref{eq:crossing}.
\end{remark}

 To achieve this theorem, we first establish the correct probabilistic setting.  The authors of \cite{BKLMR} were able to use a simpler result on randomness because in their model for 2-bridge knots, the orientations before each pair of twist regions are fixed throughout.  We are not so lucky; instead our orientations are properly modeled by a Markov chain.

\begin{definition}
\label{def:random}
The set $OW(s)$ of oriented words of size $s$ can be partitioned into the set $OW^{Amp}(s)$ of all the words of amphichiral type and the set $OW^{Chi}(s)$ of all words that are of chiral type.

Each oriented word $\vec{\omega}=(o,z)\in OW(s)$ corresponds to an oriented braid where $o$ is the orientation that occurs to the left of braid word $z$ and $o'$ is the orientation to the right of $z$.  Let $\mirror(\vec{\omega})$ be the mirror of $\vec{\omega}$ as given in Definition \ref{def:mirror}.

Now we introduce our random variables.  Let $Z_1,Z_2,\ldots$ be a sequence of independent random variables, each chosen uniformly from $W(s)$.  
Let $X_1$ denote the orientation immediately to the right of the first crossing; this orientation is always $o_1$. 
 Let $X_{\ell+1}$ denote the orientation immediately after $Z_\ell$. Since $X_{\ell+1}$ is determined by the $\ell$th oriented word $(X_\ell,Z_\ell)$, it is a random quantity.
\end{definition}

Next we establish the transition matrix $\mathbf{P}$ for the Markov chain and find a formula for $\mathbf{P}^k$.

\begin{proposition}
Let $s$ be a positive integer.  By the construction given in Definition \ref{def:random}, the sequence $\{X_i\}$ is a Markov chain on the set $OR$.  
Using the ordering  $o_1,o_2,o_3$ of the set of states,  its transition probability matrix $\mathbf{P}$ is
\begin{equation}
   \label{eq:Pmatrix}
       \mathbf{P}   \; = \;  \frac{1}{2^s}    \left(      \begin{array}{ccc}  1 & 1 & 0 \\  1 & 0 & 1 \\ 0 & 1 & 1 \end{array} \right)^s .
\end{equation}
Moreover, for any nonnegative integer $k$, the $k$th power
$\mathbf{P}^k$ equals
\begin{eqnarray}
         \label{eq:Pmatrix2odd}
       & &    \frac{1}{3}  \left(      \begin{array}{ccc}  1+2^{-ks} & 1+2^{-ks} & 1-2^{1-ks}  \\ 
       1+2^{-ks} & 1-2^{1-ks} & 1+2^{-ks} \\
              1-2^{1-ks} & 1+2^{-ks}    &      1+2^{-ks}  \end{array} \right)   
              \hspace{5mm} \hbox{if $ks$ is odd, and }
        \\
        \label{eq:Pmatrix2even}
        & &  \frac{1}{3}  \left(      \begin{array}{ccc}  
        1+2^{1-ks} & 1-2^{-ks} & 1-2^{-ks}  \\  1-2^{-ks} & 1+2^{1-ks} & 1-2^{-ks} \\
              1-2^{-ks} & 1-2^{-ks}    &      1+2^{1-ks}  \end{array} \right)   
              \hspace{5mm} \hbox{if $ks$ is even.}
\end{eqnarray}
\end{proposition}

The entries of the matrix $\mathbf{P}$ are the transition probabilities,
\[   \mathbf{P}_{ij} \;\equiv\;  \mathbf{P}(o_i,o_j)  \;=\; \Pr( X_{\ell+1}=o_j \,|\, X_\ell=o_i\, ) \,,
\]
and a well-known property of Markov chains tells us that the powers of $\mathbf{P}$ give the multi-step transition probabilities:
\[    (\mathbf{P}^k)(o_i,o_j)   \;=\; \Pr( X_{t+k}=o_j \,|\, X_t=o_i\,) \,.
\]
As $s$ gets larger, Equations \eqref{eq:Pmatrix2odd} and \eqref{eq:Pmatrix2even} show that the probabilities of each of the orientations approach $\frac{1}{3}$.

\begin{proof}
We first prove Equation \eqref{eq:Pmatrix}.  For $s=1$, the set $W(1)=\{\sigma_1,\sigma_2^{-1}\}$.  Consider the map from the braid group to the symmetric group sending each generator $\sigma_i^{\pm1}$ to the transposition $(i \; i+1)$.  Then the first element of $W(1)$ exchanges the lower two strands and the second exchanges the upper two strands.  Divide by the number of elements of $W(1)$.  To get the matrix for general $\mathbf{P}$, simply multiply this matrix by itself $s$ times.

To get Equations 
\eqref{eq:Pmatrix2odd} and \eqref{eq:Pmatrix2even}, 
let $B=  \left(      \begin{array}{ccc}  1 & 1 & 1 \\  1 & 1 & 1 \\ 1 & 1 & 1 \end{array} \right)$ 
and $J= \left(      \begin{array}{ccc}  0 & 0 & 1 \\  0 & 1 & 0 \\ 1 & 0 & 0 \end{array} \right)$.  
Then $M=B-J$.  
Let $r=ks$.  Observe that $J^2=I$, $B^2=3B$, and $JB=B=BJ$.  Then using the Binomial Theorem twice, we 
find that
\begin{eqnarray*}
    M^r  \; = \; (B-J)^r  &=& (-J)^r \,+\,\sum_{k=1}^r \binom{r}{k} B^k (-1)^{r-k}  \\
   & = & (-J)^r  \,+\,  \sum_{k=1}^r \binom{r}{k} (-1)^{r-k} 3^{k-1}B  \\
   & = & (-J)^{r} \,+\,  (-1)^{r+1}\left(\frac{  \sum_{k=0}^r\binom{r}{k}(-3)^k  \,-\,1 }{-3}\right) B 
   \\
   & = & (-J)^r  \,+\, (-1)^{r+1}\left( \frac{(1-3)^r-1}{-3} \right)\,B
   \\
   & = & (-J)^r  \,+\, \frac{1}{3}\left(2^r-(-1)^r\right)\,B\,.
\end{eqnarray*}
Thus we have that $M^r$ equals $-J+ \frac{1}{3}(2^r+1)B$  if $r$ is odd, and $I+\frac{1}{3}(2^r-1)B$ if $r$ is even.  Since $P^k=2^{-r}M^r$, Equations \eqref{eq:Pmatrix2odd} and 
\eqref{eq:Pmatrix2even} follow.
\end{proof}
It follows from Equations (\ref{eq:Pmatrix2odd}) and (\ref{eq:Pmatrix2even}) that
\begin{equation}
   \label{eq.pkdiffbound}
   \left| (\mathbf{P}^{k})(o_i,o_j)  \,-\,  (\mathbf{P}^{k})(o_i,o_\ell) \right|   \;\leq \;(2^{-s})^k
   \hspace{5mm}\hbox{for all }o_i,o_j,o_\ell \in OR.
\end{equation}

Now we introduce our word-counting process.

\begin{definition}
\label{def:indicator}
Partition the set $OW^{Chi}(s)$ of words  into 
two sets $OW^{+}(s)$ and $OW^{-}(s)$ by arbitrarily assigning one of $\vec{\omega}$ or $\mirror(\vec{\omega})$ to be in $OW^{+}(s)$ and the other to be in $OW^{-}(s)$.

For each word of chiral type $\vec{\omega}\in OW^{+}(s)$ and $\ell\in \mathbb{N}$, define the plus/minus indicator of
the word $\vec{\omega}$ at step $\ell$ to be 
\[     I_{\vec{\omega}}(\ell)  \; = \; \begin{cases}   +1 & \hbox{if }(X_{\ell},Z_\ell) \,=\, \vec{\omega}  \\   -1 & \hbox{if }(X_{\ell},Z_\ell) \,=\, \mirror(\vec{\omega})  
      \\   0 & \hbox{otherwise}    \end{cases}
\]  
and the displacement from the origin in the $\vec{\omega}$ dimension after $t$ steps to be
\[  D_{\vec{\omega}}(t) \;=\;  \sum_{\ell=1}^t I_{\vec{\omega}}(\ell) \,.
\]

Analogously, for a word of amphichiral type $\vec{\omega}\in OW^{Amp}(s)$ and $\ell\in \mathbb{N}$, define 
\[     I_{\vec{\omega}}(\ell)  \; = \; \begin{cases}   +1 & \hbox{if }(X_{\ell},Z_\ell) \,=\, \vec{\omega} 
      \\   0 & \hbox{otherwise}    \end{cases}
\]  
and 
\[  D_{\vec{\omega}}(t) \;=\;  \sum_{\ell=1}^t I_{\vec{\omega}}(\ell) 
  \mod 2\,.
\]

Lastly, define the taxicab distance from the origin of the word-counting process at time $t$ to be 
\begin{equation}  
 \label{eq.defDist}
 \text{Dist}(t)   \;=\;  \sum_{\vec{\omega}\in\, OW^{+}(s)\cup OW^{Amp}(s)}|D_{\vec{\omega}}(t)| \,.
\end{equation}
\end{definition}

Observe that $\sum_{\vec{\omega}\in\, OW^{+}(s)\cup OW^{Amp}(s)}|I_{\vec{\omega}}(\ell)|$ equals 1 for every $\ell\in \mathbb{N}$.  
We view the word-counting process
as taking values in a high-dimensional space where each coordinate is indexed by a word in $OW^{+}(s)\cup OW^{Amp}(s)$.

Since the expectation
\[  E(\text{Dist}(t))  =  \sum_{\vec{\omega}\in \,OW^{+}(s)}E|D_{\vec{\omega}}(t)| 
  \,+\,\sum_{\vec{\omega}\in \, OW^{Amp}(s)}E|D_{\vec{\omega}}(t)|, 
\]
our next step in finding an upper bound on $E(\text{Dist}(t))$ in Theorem \ref{thm:taxicab} is to find an upper bound on the expectation $E|D_{\vec{\omega}}(t)|$.

\begin{proposition}
   \label{prop.Jbound}
Let $s$ and $t$ be positive integers, and let $\vec{\omega}\in OW^{Chi}(s)$. The expected absolute value of the displacement from the origin in the $\vec{\omega}$ coordinate of the word-counting process after $t$ steps satisfies 
$E| D_{\vec{\omega}}(t)|   \,\leq \,  2\sqrt{t/2^s}$.   
\end{proposition}
\begin{proof}
For our fixed $\vec{\omega}\in OW^{Chi}(s)$, we shall write
 $\vec{\omega}$ as $(o_*,z_*)$ and 
 $\mirror(\vec{\omega})$ as $(\overline{o}_*,\overline{z}_*)$.  
Recall that we write $o_*'$ for the orientation that occurs immediately after $\vec{\omega}$, and then $\overline{o}_*'$ is the orientation that occurs immediately after $\mirror(\vec{\omega})$.

Since $\left(E|D_{\vec{\omega}}(t)|\right)^2 \;\leq \;  E(D_{\vec{\omega}}(t)^2)$ by the Cauchy-Schwarz inequality, our goal shall be to obtain a bound on $E(D_{\vec{\omega}}(t)^2)$.  We use 
\begin{equation}
   \label{eq.JsumI}
     E(D_{\vec{\omega}}(t)^2)  \;=\;  \sum_{\ell=1}^t E(I_{\vec{\omega}}(\ell)^2)  \,+\,  2 \sum_{k=1}^{t-1} \sum_{\ell=1}^{t-k}  E( \,I_{\vec{\omega}}(\ell)\,I_{\vec{\omega}}(\ell+k)\, ) \,.
\end{equation}
We first consider $E(I_{\vec{\omega}}(\ell)^2)$.  Observe that $I_{\vec{\omega}}(\ell)$ equals 0 unless $(X_\ell,Z_\ell)$ is either $(o_*,z_*)$ or $(\overline{o}_*,\overline{z}_*)$, in which 
case $I_{\vec{\omega}}(\ell)^2$ equals 1.
Thus
\begin{eqnarray} 
    \nonumber
    E(I_{\vec{\omega}}(\ell)^2)  &=&  \Pr(X_\ell=o_*,Z_\ell=z_*)+\Pr(X_\ell=\overline{o}_*,Z_\ell=\overline{z}_*)   \\
    \label{eq.I2bound}
  & \leq &    \Pr(Z_\ell=z_*)+\Pr(Z_\ell=\overline{z}_*) \qquad\;=\;  \frac{2}{2^{s}}.
\end{eqnarray}

Next  we consider the term $ E( \,I_{\vec{\omega}}(\ell)\,I_{\vec{\omega}}(\ell+k)\, )$, where $k\geq 1$.
Observe that $I_{\vec{\omega}}(\ell)I_{\vec{\omega}}(\ell+k)$ equals 0 unless $(X_\ell,Z_\ell)$ and $(X_{\ell+k},Z_{\ell+k})$ each is either $(o_*,z_*)$ or $(\overline{o}_*,\overline{z}_*)$, in which 
case $I_{\vec{\omega}}(\ell)I_{\vec{\omega}}(\ell+k)$ equals $1$ or $-1$.
Thus 
\begin{equation}
    \label{eq.Eaaaa}
     E(\,I_{\vec{\omega}}(\ell)\,I_{\vec{\omega}}(\ell+k)\, )   \;=\;  a(1,1) \,-\,a(1,-1) \,+\,a(-1,-1)\,-\,a(-1,1) \,,
\end{equation}
where
\[    a(i,j)   \;=\;  \Pr(I_{\vec{\omega}}(\ell)=i, \,I_{\vec{\omega}}(\ell+k)=j)\,.
\]

We will need to exploit some cancellation here.  Let's start with $a(1,1)-a(1,-1)$.  We have
\[   a(1,1)  = \Pr(X_\ell=o_*,Z_\ell=z_*,X_{\ell+k}=o_*,Z_{\ell+k}=z_*)  
        = u_1 u_2 u_3 u_4\,,   
\]
where
\begin{eqnarray*}
u_1\; & =\; & \Pr(X_\ell=o_*),   \\
u_2 \; & =\; &  \Pr(Z_{\ell}=z_*\,|  \,  X_\ell=o_*)  \\  u_3  \; & =\; &  \Pr(X_{\ell+k}=o_*\,|  \,  X_\ell=o_*,Z_\ell=z_*)   
        \\   
u_4 \; & =\; & \Pr(Z_{\ell+k}=z_*\,|  \,  X_\ell=o_*,Z_\ell=z_*,X_{\ell+k}=o_*)   \,.
\end{eqnarray*}
By the independence of terms of the $Z$ sequence, we have $u_2=u_4 = 2^{-s}$.  
The value of $u_1$ depends on the initial orientation $X_1=o_1$, via
\[   u_1 \;=\;  \Pr(X_\ell=o_*)  \;=\;  
  \mathbf{P}^{\ell-1}(o_1,o_*) \,.
\]
The quantity $u_3$ is a bit more subtle.   Given the information that $X_\ell=o_*$ and $Z_\ell=z_*$, it must also be 
true that $X_{\ell+1}=o_*'$.  Therefore, the Markov property of the $X$ process tells us that 
\[  u_3 \;=\;  \Pr(X_{\ell+k}=o_*\,|  \,  X_{\ell+1}=o_*') \;=\;  (\mathbf{P}^{k-1})(o_*',o_*) \,.
\]
Putting the pieces together shows that 
\[     a(1,1)  \;=\;  (2^{-s})^2 \,\Pr(X_\ell=o_*)  \,(\mathbf{P}^{k-1})(o_*',o_*) \,.
\]
An exactly analogous argument (with the only difference being in the $u_3$ term) shows that 
\begin{eqnarray*}
    a(1,-1)  & = & \Pr(X_\ell=o_*,Z_\ell=z_*,X_{\ell+k}=\overline{o}_*,Z_{\ell+k}=\overline{z}_*)  \\
       & = & (2^{-s})^2 \,\Pr(X_\ell=o_*)  \,(\mathbf{P}^{k-1})(o_*',\overline{o}_*) \,.
\end{eqnarray*}
Then 
\begin{eqnarray}
   \nonumber
   |a(1,1)-a(1,-1)|  &=  & (2^{-s})^2 \,\Pr(X_\ell=o_*)  \,\left| (\mathbf{P}^{k-1})(o_*',o_*) \,-\,   (\mathbf{P}^{k-1})(o_*',\overline{o}_*)\right|
     \\
     \label{eq.adiff1}
     & \leq &  (2^{-s})^2 \cdot 1 \cdot (2^{-s})^{k-1}        \hspace{12mm}\hbox{(by Eq.\ (\ref{eq.pkdiffbound}))}.
\end{eqnarray}
Again, a very similar calculation shows that
\begin{equation}
           \label{eq.adiff2}
    |a(-1,-1)\,-\,a(-1,1)|  \;\leq \;  (2^{-s})^2 \cdot 1 \cdot (2^{-s})^{k-1}.
\end{equation}
Inserting Equations (\ref{eq.adiff1}) and (\ref{eq.adiff2}) into (\ref{eq.Eaaaa}) shows that 
\begin{equation}
   \label{eq.EIIbound}
      \left|   E(\,I_{\vec{\omega}}(\ell)\,I_{\vec{\omega}}(\ell+k)\, ) \right|  \;\leq \;  2 \,(2^{-s})^{k+1} \hspace{5mm} \hbox{for all $k\geq 1$}.
\end{equation}
Finally, using Equations (\ref{eq.JsumI}), (\ref{eq.I2bound}),     and (\ref{eq.EIIbound}), we obtain
\begin{eqnarray}
  \nonumber
   E(D_{\vec{\omega}}(t)^2) & \leq &   2t \,2^{-s} \,+\,  2\sum_{k=1}^{t-1}  (t-k)  \,(2^{-s})^{k+1}
   \\
    \nonumber
   & \leq & 2t  \left(2^{-s}  \,+  \,\frac{(2^{-s})^2}{1-2^{-s}} \right)
   \\
      \label{eq.J2sumbound}
   &\leq & 4t\,2^{-s}   \hspace{10mm}\hbox{(since $2^{-s}\leq 1/2$)}.
\end{eqnarray}
The proposition now follows because $E|D_{\vec{\omega}}(t)| \;\leq \;  \sqrt{E(D_{\vec{\omega}}(t)^2)}$.
\end{proof}

We can now finish proving the main theorem from this subsection.

\begin{proof}[Proof of Theorem \ref{thm:taxicab}]
Proposition \ref{prop.Jbound} is the main ingredient for proving the bound on the expected value of the taxicab distance $\text{Dist}(t)$ as defined in Equation
(\ref{eq.defDist}).  Recall that $a=|OW^{Amp}(s)|$ is the number of oriented words of amphichiral type. 
Then the number of oriented word-mirror pairs is $|OW^{+}(s)|=\frac{1}{2}(3\cdot 2^s-a)$.
Since $|D_{\vec{\omega}}(t)|\leq 1$ when $\vec{\omega}\in OW^{Amp}(s)$,
the expected taxicab distance of the word-counting process after $t$ steps satisfies
\begin{eqnarray*}
  E(\text{Dist}(t)) & = & \sum_{\vec{\omega}\in \,OW^{+}(s)}E|D_{\vec{\omega}}(t)| 
  \,+\,\sum_{\vec{\omega}\in \, OW^{Amp}(s)}E|D_{\vec{\omega}}(t)| 
  \\ &  \leq  & 
    \frac{3\cdot 2^s-a}{2} \left(  2  \sqrt{\frac{t}{2^s}} \right) \,+\, a\cdot 1
    \\  & \leq & 
    3 \sqrt{2^s\,t} \,+\,a\, ,
\end{eqnarray*}
where the second line above follows because $|OW^+(s)|=\frac{3\cdot 2^s-a}{2}$ and $|OW^{Amp}(s)|=a$. 
\end{proof}

Finally in the last subsection, we will analyze all of the contributions to the upper bound of the expected 4-genus.

\subsection{Analysis}  
\label{subsec:analysis}  Our last goal is to determine the upper bound for the expected 4-genus of $\frac{9.75c}{\log c}$ given in Theorem \ref{thm:avg4genus}.  First we consider contributions from $T(2m+1)$ and $T(2m+2)$ together.

\begin{proposition}
    \label{prop:avg4genus2}
    The average $4$-genus over all knots in $T_m=T(2m+1)\cup T(2m+2)$ satisfies the inequality
    \[\frac{\sum_{w\in T_m} g_4(K_w)}{|T_m|} < \frac{4.94c}{\log c},\]
    where the base of the logarithm is 10 and $c=2m+j$ for $j=1$ or $2$.
\end{proposition}

\begin{proof} Proposition \ref{prop:newii} and Remark \ref{rem:cutting} construct a cobordism of genus at most $t+1$ between $K_w$ and a connected sum $L_1\#\cdots \# L_{t+1}$ of $t+1$ knots and links. In Lemma \ref{lem:link} and Remark \ref{rem:link}, we use the operation $\merge$ to construct a cobordism of genus at most $\frac{3}{2}t+\frac{1}{2}$ between $L_1\#\cdots \# L_{t+1}$ and a connected sum of knots $K_1\#\cdots \#K_{t+1}$. Theorem \ref{thm:taxicab} and Remark \ref{rem:randomwalkdistance} show that the expected $4$-genus of $K_1\#\cdots \#K_{t}$ is at most $\frac{1}{2}(s+3)(3\sqrt{2^s t}+a)$, where $a\leq 2^{\frac{s}{2}}$. Inequality \ref{eq:crossing} implies that the $4$-genus of the last summand $K_{t+1}$ is at most $\frac{1}{2}(r-1)$.

Thus an upper bound for the average 4-genus of $K_w$ where $w\in T_m$ is given by
\begin{equation}
\label{eq:analysis}
(t+1) +\left(\frac{3}{2}t +\frac{1}{2}\right)+\frac{3}{2}{(s+3)}\sqrt{2^st}+\frac{1}{2}{(s+3)}2^{\frac{s}{2}}+\frac{1}{2}{(r-1)},
\end{equation}
where $2m+j=c=st+r$ and $1+j\leq r\leq s+j$.

We give a simpler upper bound for this expression.  To achieve this we let $s$ grow towards infinity by setting $s=\lceil \log c \rceil$ with logarithm base 10.  Then $s-1 < \log c \leq s$ or equivalently, $10^{s-1}< c \leq 10^{s}$. Since $t=\frac{c}{s}-\frac{r}{s}<\frac{c}{s}\leq\frac{c}{\log c}$, it follows that $t<\frac{c}{\log c}$. 

Inequality \ref{eq:crossing} implies that $g_4(K) < \frac{c}{2}$ for any knot $K$, and thus $\frac{\sum_{w\in T_m} g_4(K_w)}{|T_m|} < \frac{c}{2}$. Since $\frac{c}{2}\leq \frac{4.94c}{\log c}$ for all $c\leq 7\times 10^{9}$, we only need to prove our inequality for larger $c$. We make a less extreme assumption on $c$. For the remainder of the proof, suppose $c> 10000$. Then $s\geq 5$ and $t\geq 1999.$

The sum of the terms in the first two parentheses in Expression \eqref{eq:analysis} is  $\frac{5t+3}{2}$, and since $t\geq 1999$, it follows that 
\begin{equation}
\label{eq:CS_LL}
    (t+1)+\left(\frac{3}{2}t+\frac{1}{2}\right)=\frac{5t+3}{2} < \frac{251t}{100} = 2.51 t < \frac{2.51 c}{\log c}.
\end{equation}  

In order to bound the next summand $\frac{3}{2}{(s+3)}\sqrt{2^st}$ in Expression \eqref{eq:analysis}, we first observe that  $2s^3<5^{s-1}$ for ${s}\geq 5$.  Since $c>10^4$ and $s\geq 5$, we have
\[
 s^2 2^{s} = \frac{2s^32^{s-1}}{s} < \frac{5^{s-1}2^{s-1}}{s} = \frac{10^{s-1}}{s}<\frac{c}{\log c}.
\]
Because $t< \frac{c}{\log c}$ and $s+3\leq \frac{8}{5}s$ for $s\geq 5$, it follows that
\begin{equation}
\label{eq:end}
\frac{3}{2}(s+3)\sqrt{2^s t} \leq \frac{12}{5}\;s\sqrt{2^s t}=\frac{12}{5}\sqrt{s^2 2^s t}<\frac{12}{5}\sqrt{\left(\frac{c}{\log c}\right)^2}=\frac{12 c}{5\log c} =\frac{2.4 c}{\log c}.
\end{equation}

The upper bound on the following summand $\frac{1}{2}{(s+3)}2^{\frac{s}{2}}$ of \eqref{eq:analysis} follows from a similar argument as the previous one. Because $s\geq 5$, we have
\[\frac{1}{2}(s+3)2^{\frac{s}{2}} \leq \frac{4}{5} s 2^{\frac{s}{2}} = \frac{4}{5} \sqrt{s^2 2^s} \leq \frac{4}{5} \sqrt{\frac{c}{\log c}}.\]
Since $\sqrt{x}\leq \frac{x}{\sqrt{x_0}}$ for all $x\geq x_0$ and $50<\sqrt{\frac{10,001}{\log 10,001}}$, it follows that
\begin{equation}
\label{eq:atermbound}    
\sqrt{\frac{c}{\log c}} \leq \frac{c}{50\log c} = \frac{0.02c}{\log c}
\end{equation}
for all $c> 10^4$.

For the final summand $\frac{1}{2}{(r-1)}$ of Expression \eqref{eq:analysis}, since $r-1-j\leq s -1$, we have that $r\leq s+j \leq s+2$, and so $\frac{1}{2}(r-1)\leq \frac{s}{2}+\frac{1}{2}$.  Since for $s\geq 4$ we have $50s(s+1)\leq 10^{s-1} < c$, we have that
\begin{equation}
\label{eq:lastsummandbound}    
\frac{r-1}{2}\leq \frac{s+1}{2}\leq \frac{1}{100} \frac{10^{s-1}}{s} < \frac{c}{100\log c} = \frac{0.01c}{\log c}.
\end{equation}

Adding the upper bounds on each of the summands of Expression \eqref{eq:analysis} from Inequalities \eqref{eq:CS_LL} through \eqref{eq:lastsummandbound} yields an overall upper bound of
\[
\frac{\sum_{w\in T_m} g_4(K_w)}{|T_m|}< \frac{2.51c}{\log c}
+
\frac{2.4c}{\log c}
+
\frac{0.02c}{\log c}
+
\frac{0.01c}{\log c}
<
\frac{4.94c}{\log c},
\]
completing the proof.
\end{proof}

We use the bound on the average $4$-genus of knots from $T(2m+1)\cup T(2m+2)$ in Proposition \ref{prop:avg4genus2} to prove Theorem \ref{thm:avg4genus}, providing an explicit sublinear upper bound on the average 4-genus of the set of 2-bridge knots with crossing number $c$.

\begin{proof}[Proof of Theorem \ref{thm:avg4genus}]
We first convert the estimate from Proposition \ref{prop:avg4genus2}, which averages over $T_m=T(2m+1)\cup T(2m+2)$,
into an estimate for a single crossing number $c$.

Let $c=2m+j$, where $j\in\{1,2\}$, and write $c=st+r$ with $1+j\leq r\leq s+j$. In Proposition \ref{prop:avg4genus2}, the only part of the estimate that genuinely uses the average over the two rows $T(2m+1)\cup T(2m+2)$ is Inequality \eqref{eq:end}, which uses the random-walk estimate from Theorem \ref{thm:taxicab}. In particular, Theorem \ref{thm:taxicab} bounds the expected number of
chiral summands not canceled by their mirrors over all of $T_m$ by
$3\sqrt{2^s t}$. Thus the total number of such summands over $T_m$ is
bounded above by
\[
|T_m|\,3\sqrt{2^s t}=2^{2m-1}\,3\sqrt{2^s t}.
\]

We now pass from $T_m$ to a single row $T(c)$. If $c=2m+1$, then, in the
worst case, all of these non-canceling summands occur in $T(2m+1)$. Hence
their expected number over $T(c)$ is bounded above by
\[
\frac{|T_m|\,3\sqrt{2^s t}}{t(2m+1)}
=
\frac{2^{2m-1}\,9\sqrt{2^s t}}{2^{2m-1}+1}
<
9\sqrt{2^s t}.
\]
If $c=2m+2$, then the same worst-case argument gives
\[
\frac{|T_m|\,3\sqrt{2^s t}}{t(2m+2)}
=
\frac{2^{2m-1}\,9\sqrt{2^s t}}{2^{2m}-1}
<
9\sqrt{2^s t}.
\]
Therefore, for either parity of $c$, the expected number of non-canceling
chiral summands in $T(c)$ is at most $9\sqrt{2^s t}$. Since each such summand has at most $s+3$ crossings, Inequality
\eqref{eq:crossing} gives a contribution of at most $\frac{9}{2}(s+3)\sqrt{2^s t}$
to the average 4-genus over $T(c)$. Multiplying Inequality \eqref{eq:end} yields its replacement
\begin{equation}
\label{eq:onec}
\frac{9}{2}(s+3)\sqrt{2^s t} < \frac{36c}{5\log c} = \frac{7.2 c}{\log c}.
\end{equation}

The estimates from Proposition \ref{prop:avg4genus2} for the remaining
terms in Expression \eqref{eq:analysis} do not depend on averaging over both
rows. Thus the upper bounds in Inequalities \eqref{eq:CS_LL}, \eqref{eq:onec}, \eqref{eq:atermbound}, and \eqref{eq:lastsummandbound}   add to give the following upper bound:
\begin{equation}
\label{eq:staticterms}
\frac{\sum_{w\in T(c)} g_4(K_w)}{t(c)}
<
\frac{2.51c}{\log c}
+
\frac{7.2c}{\log c}
+
\frac{0.02c}{\log c}
+
\frac{0.01c}{\log c}
<
\frac{9.74c}{\log c}.
\end{equation}

It remains to pass from the partially double-counted set $T(c)$ to the set
$K(c)$ of 2-bridge knots, where chiral pairs are counted only once. Recall
that a non-palindromic knot is represented by two words in $T(c)$, while a
palindromic knot is represented by only one word. Therefore $2|K(c)|=t(c)+t_p(c)$, and
\begin{align*}\overline{g_4}(c) = &\; \frac{\sum_{w\in T(c)} g_4(K_w) + \sum_{w\in T_p(c)} g_4(K_w)}{2|\mathcal{K}(c)|}\\
= & \; \frac{\sum_{w\in T(c)} g_4(K_w) + \sum_{w\in T_p(c)} g_4(K_w)}{t(c) + t_p(c)}\\
\leq & \;\frac{\sum_{w\in T(c)} g_4(K_w)}{t(c)}  + \frac{\sum_{w\in T_p(c)} g_4(K_w)}{t(c)}\\
< & \; \frac{9.74 c}{\log c} + \frac{c t_p(c)}{2t(c)},
\end{align*}
where the first inequality follows because $t(c)+t_p(c)\geq t(c)$ and the second inequality from both Inequalities \eqref{eq:staticterms} and \eqref{eq:crossing}. 
A computation shows that for all $c\geq 14$,
\[\frac{c t_p(c)}{2t(c)} = \frac{c\left(2^{\left\lfloor\frac{c-1}{2}\right\rfloor}-(-1)^{\left\lfloor\frac{c-1}{2}\right\rfloor}\right)}{2\left(2^{c-2}-(-1)^{c}\right)} < \frac{0.01c}{\log c},\]
implying that $\overline{g_4}(c) < \frac{9.75 c}{\log c}$ when $c\geq 14$.

For $3\leq c<14$, the bound $g_4(K)<c/2$ gives
\[
\overline{g_4}(c)<\frac{c}{2}<\frac{9.75c}{\log c}.
\]
Thus the desired bound holds for all $c\geq 3$. Since $\frac{c}{\log c}$ is sublinear in $c$, the average 4-genus of a $2$-bridge knot is sublinear in the crossing number.
\end{proof}

\begin{remark}
In \cite{CohLow} the first and second authors present the average 3-genus for 2-bridge knots.  This result could be implemented above to reduce the coefficient further, but the error term of the result would make our proof more complicated.
\end{remark}

\newcommand{\etalchar}[1]{$^{#1}$}
\def\cprime{$'$}
\providecommand{\bysame}{\leavevmode\hbox to3em{\hrulefill}\thinspace}

\end{document}